\documentclass[11pt]{article}
\usepackage{amssymb,amsmath,bm,mathrsfs,makeidx,amsfonts,graphicx,amsthm}
\usepackage[authoryear,nonamebreak,bibstyle]{natbib}
\usepackage{epsfig}
\usepackage{color}

\newtheorem{thm}{Theorem}

\newtheorem{lem}{Lemma}

\newtheorem{rmk}{Remark}

\newtheorem{assu}{Assumption}
\numberwithin{equation}{section}
\def\beginn{\begin{eqnarray*}}
\def\endn{\end{eqnarray*}}
\def\beginy{\begin{eqnarray}}
\def\endy{\end{eqnarray}}
\def\begine{\begin{enumerate}}
\def\ende{\end{enumerate}}

\def\be{\begin{equation}}
\def\ee{\end{equation}}
\def\bea{\begin{eqnarray}}
\def\eea{\end{eqnarray}}

\def \im {{\rm Im}}
\newcommand{\non}{\nonumber \\}
\newcommand{\bbX}{{\bf X}}
\newcommand{\bbx}{{\bf x}}

\newcommand{\bby}{{\bf y}}
\newcommand{\bbZ}{{\bf Z}}
\newcommand{\bbz}{{\bf z}}
\newcommand{\bbw}{{\bf w}}
\newcommand{\bbA}{{\bf A}}

\newcommand{\bbB}{{\bf B}}
\newcommand{\bbC}{{\bf C}}

\newcommand{\bbD}{{\bf D}}

\newcommand{\bbe}{{\bf e}}

\newcommand{\bbH}{{\bf H}}
\newcommand{\bbh}{{\bf h}}
\newcommand{\bbI}{{\bf I}}

\newcommand{\bbQ}{{\bf Q}}

\newcommand{\bbr}{{\bf r}}
\newcommand{\bbs}{{\bf s}}
\newcommand{\bbS}{{\bf S}}

\newcommand{\bbT}{{\bf T}}

\newcommand{\bbu}{{\bf u}}
\newcommand{\bbV}{{\bf V}}
\newcommand{\bbv}{{\bf v}}
\newcommand{\bbW}{{\bf W}}

{
\begin{document}

\title{Independence Test for High Dimensional Random Vectors}
\author{G. Pan, J. Gao\footnote{{\rm Correspondence}: Jiti Gao, Department of Econometrics and Business Statistics, Monash University, Caulfield East Victoria 3145, Australia. Email: jiti.gao@monash.edu.}, Y. Yang and M. Guo}

\maketitle

\begin{abstract}
This paper proposes a new mutual independence test for a large number of high dimensional random vectors. The test statistic is based on the characteristic function of the empirical spectral distribution of the sample covariance matrix. 
 The asymptotic distributions of the test statistic under the null and local alternative hypotheses are established as dimensionality and the sample size of the data are comparable. We apply this test to examine multiple MA(1) and AR(1) models, panel data models with some spatial cross-sectional structures.
  In addition, in a flexible applied fashion, the proposed test can capture some dependent but uncorrelated structures, for example, nonlinear MA(1) models, multiple ARCH(1) models and vandermonde matrices.
 Simulation results are provided for detecting these dependent structures. An empirical study of dependence between closed stock prices of several companies from New York Stock Exchange (NYSE) demonstrates that the feature of cross--sectional dependence is popular in stock markets. \\

\end{abstract}

\section{Introduction}
A prominent feature of data collection nowadays is that the number of variables is comparable with the sample size. This is the opposite of the classical situations where many observations are made on low-dimensional data. Specific examples of such high dimensional data include microarray expression, images, multiuser detection, climate data and financial data (see, for example, \cite{Do2000}, \cite{Fan2011}, \cite{Johnstone2001}, \cite{Johnstone2009}). This type of data trend poses great challenges because traditional multivariate approaches do not necessarily work, which were established for the case of the sample size $n$ tending to infinity and the dimension $p$ remaining fixed (See \cite{Ader1984}). There have been a substantial set of research works dealing with high dimensional data (see, for example, \cite{BS1996}, \cite{Fan2010}, \cite{Huang2008}, \cite{Fan2001}). Measuring mutual dependence is important in time series analysis and cross-sectional panel data analysis. While serial dependence can be characterized by the general spectral density function (see \cite{Hong1998}; \cite{Hong1999}), mutual dependence is difficult to be described by a single criteria. This paper proposes a new statistic to test mutual dependence for a large number of high dimensional random vectors, including multiple time series and cross-sectional panel data.
\medskip

Suppose that $\{X_{ji}, j=1,\ldots,n; i=1,\ldots,p\}$ are real--valued random variables. For $1\leq i\leq p$, let $\bbx_i=(X_{1i},\cdots,X_{ni})^{T}$ denote the $i$--th time series and $\bbx_1,\cdots,\bbx_p$ be a panel of $p$ time series, where $n$ usually denotes the sample size in each of the time series data. In both theory and practice, it is not uncommon to assume that each of the time series $(X_{1i}, X_{2i}, \cdots, X_{ni})$ is statistically {\color{red} {independent}}, but it may be unrealistic to assume that $\bbx_1, \bbx_2, \cdots, \bbx_p$ are independent or even uncorrelated. This is because there is no natural ordering for cross--sectional indices. There are such cases in various disciplines. In economics and finance, for example, it is not unreasonable to expect that there is significant evidence of cross--sectional dependence in output innovations across $p$ countries and regions in the world. In the field of climatology, there is also some evidence to show that climatic variables in different stations may be cross--sectionally dependent and the level of cross--sectional dependence may be determined by some kind of physical distance. Moreover, one would expect that climatic variables, such as temperature and rainfall variables, in a station in Australia have higher--level dependence with the same type of climatic variables in a station in New Zealand than those in the United States.
\medskip

In such situations, it may be necessary to test whether $\bbx_1, \bbx_2, \cdots, \bbx_p$ are independent before a statistical model is used to model such data. In the econometrics and statistics literature, several papers have basically considered testing for cross--sectional uncorrelatedness for the residuals involved in some specific regression models. Such studies include \cite{pesaran2004} for the parametric linear model case, \cite{HPP2007} for the parametric nonlinear case, and \cite{CGL2011} for the nonparametric nonlinear case. Other related papers include \cite{SU2009} for testing conditional uncorrelatedness through examining a covariance matrix in the case where $p$ is fixed. As the main motivation of this paper, we will propose using an empirical spectral distribution function based test statistic for cross--sectional independence of $\bbx_1, \bbx_2, \cdots, \bbx_p$.
\medskip

The aim is to test
\begin{equation}\label{a1}
\textbf{H}_0: \bbx_1,\cdots,\bbx_p \  are \ independent; \ \text{against} \ \textbf{H}_1: \bbx_1,\cdots,\bbx_p \ are \ not \ independent,
\end{equation}
where $\bbx_i=(X_{1i},\ldots,X_{ni})^{T}$ for $i=1,\ldots,p$.
\medskip

In time series analysis, mutual independence test for multiple time series has long been of interest. Moreover, time series always display various kinds of dependence. For example, an autoregressive conditional heteroscedastic (ARCH(1)) model involves a martingale difference sequence (MDS); a nonlinear moving average (MA) model is not a MDS, but its autocorrelations are zero; a linear moving average (MA) model and an autoregressive (AR) model are both models with correlated structures. In this paper, we also employ the proposed statistic to test dependence for multiple time series.
\medskip

Section $8.5$ of \cite{Ader1984} also consider a similar problem but with fixed dimensions. His problem and approach are as follows. Let the $pm$-component vector $\bbx$ be distributed according to $N(\boldsymbol\mu,\boldsymbol\Sigma)$. Partition $\bbx$ into $m$ subvectors with $p$ components respectively, that is, $\left(\bbh_1^T, \cdots, \bbh_p^T\right)^T$. The vector of means $\boldsymbol{\mu}$ and the covariance matrix $\boldsymbol\Sigma$ can be partitioned respectively, $\left(\boldsymbol{\mu}_1^T, \cdots, \boldsymbol{\mu}_p^T\right)^T$, and
\begin{eqnarray}
\boldsymbol\Sigma=\left(
                    \begin{array}{cccc}
                      \boldsymbol\Sigma_{11} & \boldsymbol\Sigma_{12} & \cdots & \boldsymbol\Sigma_{1p} \\
                      \boldsymbol\Sigma_{21} & \boldsymbol\Sigma_{22} & \cdots & \boldsymbol\Sigma_{2p} \\
                      \vdots & \vdots & \ddots & \vdots \\
                      \boldsymbol\Sigma_{p1} & \boldsymbol\Sigma_{p2} & \cdots & \boldsymbol\Sigma_{pp} \\
                    \end{array}
                  \right).
\end{eqnarray}

The null hypothesis to be tested is that the subvectors $\bbh_1,\bbh_2,\ldots,\bbh_p$ are mutually independently distributed. If $\bbh_1,\bbh_2,\ldots,\bbh_p$ are independent subvectors,
\begin{eqnarray}
E(\bbx_i-\boldsymbol\mu_i)(\bbx_j-\boldsymbol\mu_j)^{'}=\boldsymbol\Sigma_{ij}=\textbf{0}, \ \forall i\neq j.
\end{eqnarray}

Thus, the null hypothesis is equivalent to testing $\bbH_0: \boldsymbol\Sigma_{ij}=\textbf{0}, \ \forall i\neq j$. This can be stated alternatively as the hypothesis that $\boldsymbol\Sigma$ is of the form
\begin{eqnarray}
\boldsymbol\Sigma_0=\left(
                    \begin{array}{cccc}
                      \boldsymbol\Sigma_{11} & \textbf{0} & \cdots & \textbf{0} \\
                      \textbf{0} & \boldsymbol\Sigma_{22} & \cdots & \textbf{0} \\
                      \vdots & \vdots & \ddots & \vdots \\
                      \textbf{0} & \textbf{0} & \cdots & \boldsymbol\Sigma_{pp} \\
                    \end{array}
                  \right).
\end{eqnarray}

To tackle the problem above, draw $n$ observations from the population $\bbx$ to form a sample covariance matrix. The likelihood ratio criterion proposed is
\begin{eqnarray}
\lambda=\frac{|\bbQ|^{\frac{1}{2}n}}{\prod^{p}_{i=1}|\bbQ_{ii}|^{\frac{1}{2}n}},
\end{eqnarray}
where
\begin{eqnarray}
\bbQ=\left(
       \begin{array}{cccc}
         \bbQ_{11} & \bbQ_{12} & \cdots & \bbQ_{1p} \\
         \bbQ_{21} & \bbQ_{22} & \cdots & \bbQ_{2p} \\
         \vdots & \vdots &  & \vdots \\
         \bbQ_{p1} & \bbQ_{p2} & \cdots & \bbQ_{pp} \\
       \end{array}
     \right),
\end{eqnarray}
in which $\bbQ_{ik}$ is the sample covariance matrix of the random vector $\bbh_{i}$ and $\bbh_k$.
\medskip

Our approach essentially uses the characteristic function of the empirical spectral distribution of sample covariance matrices in large random matrix theory.
When $\bbx_1, \cdots, \bbx_p$ are mutual independent, the limiting spectral distribution (LSD) of the corresponding sample covariance matrix is the Marcenko-Pastur (M-P) law (see, for example, \cite{MP1967}, \cite{BS2009}). From this point, any deviation of the LSD from M-P law is evidence of dependence. Indeed, \cite{S1995} and \cite{BZ2008} report the LSD of the sample covariance matrix with correlations in columns and it is different from the M-P law. In dependent but uncorrelated cases, the LSD may also be the M-P law, e.g. ARCH(1) model. However, the proposed test can still capture the dependence by utilizing the nonzero correlation between high-order series. 
Unlike the Anderson's test, we need not re--draw observations from the set of vectors of $\bbx_1,\cdots,\bbx_p$ due to the high dimensionality.
\medskip

The rest of the paper is organized as follows. Section $2$ introduces the proposed test statistic and some related large dimensional random matrix results. 
We present the asymptotic distributions of the test statistic under both the null and local alternative hypotheses in Section $3$. Moreover, its applications to some panel data models with spatial dependent structures are also illustrated. Section $4$ studies a general panel data model whose dependent structure is different from that investigated in Section $3$ and develops the asymptotic distribution of the proposed statistic under the null hypothesis for this model. Several simulated results which demonstrate the effectiveness of the proposed test under the two dependent structures are shown in Section $5$. Furthermore, this section illustrates simulation results for some dependent but uncorrelated structures, though we have not developed asymptotic theory for these kinds of models. An empirical analysis of daily closed stock prices from NYSE is provided in Section $6$. We conclude with a discussion in Section $7$. All mathematical proofs are relegated to an appendix available from a supplementary document. The full version of this paper is available as a working paper at http://www.jitigao.com.

\section{Theory and Methodology}

Before we establish the main theory and methodology, we first make the following assumptions:

\begin{assu}\label{assu1}
For each $i=1,\ldots,p$, $Y_{1i},\cdots,Y_{ni}$ are independent and identically distributed (i.i.d) random variables with mean zero, variance one and finite fourth moment. When $Y_{ji}$'s are complex random variables we require $EY_{ji}^2=0$. Let $\bbx_i=\bbT_n^{1/2}\bby_i$ with $\bby_i=(Y_{1i},\cdots,Y_{ni})^T$ and $\bbT_n^{1/2}$ being a Hermitian square root of the nonnegative definite Hermitian matrix $\bbT_n$.
\end{assu}
\begin{assu}\label{assu3}
$p=p(n)$ with $p/n\rightarrow c\in(0,\infty)$.
\end{assu}

We stack $p$ time series $\bbx_i$ one by one to form a data matrix
$\bbX=\left(\bbx_1,\cdots,\bbx_p\right)$.
Moreover denote the sample covariance matrix by
$$\textbf{A}_n=\frac{1}{n}\bbX^{*}\bbX$$ where $\bbX^{*}$ is the Hermitian transform of the matrix $\bbX$.
The empirical spectral distribution (ESD) of the sample covariance matrix $\textbf{A}_n$ is defined as
\begin{equation}
F^{\bbA_n}(x)=\frac{1}{p}\sum^{p}_{i=1}I(\lambda_i\leq x),
\end{equation}
where $\lambda_1\leq \lambda_2\leq \ldots \leq \lambda_p$ are eigenvalues of $\textbf{A}_n$.

It is well-known that if $\bbx_1,\cdots,\bbx_p$ are independent and $p/n\rightarrow c\in(0,\infty)$ then  $F^{\bbA_n}(x) $ converges with probability one to the Marcenko-Pastur Law $F^c(x)$ ( see Marcenko-Pastur (1967) and \cite{{BS2009}}) whose density has an explicit expression
\begin{equation}\label{a2}
f_{c}(x)=\left\{
\begin{array}{r@{\; \;}l}
\frac{1}{2\pi xc}\sqrt{(b-x)(x-a)}, \quad & a\leq x\leq b; \\
0,                                      \quad     & otherwise; \\
\end{array}
\right.
\end{equation}
and a point mass $1-1/c$ at the origin if $c>1$, where $a=(1-\sqrt{c})^2$ and $b=(1+\sqrt{c})^2$.
\medskip

When there is some correlation among $\bbx_1,\cdots,\bbx_p$, denote by $\bbT_p$ the covariance matrix of the first row, $\bby_1^T$, of $\bbX$. Then, under Assumption 1, when $F^{\bbT_p}(x)\stackrel{D}\longrightarrow H(x)$, $F^{\textbf{A}_n}$ converges with probability one to a non random distribution function $F^{c,H}$ whose Stieltjes transform satisfies (see \cite{S1995})
\begin{equation}\label{eq2}
m(z)=\int \frac{1}{x(1-c-czm(z))-z}dH(x),
\end{equation}
where the Stieltjes transform $m_G$ for any c.d.f $G$ is defined as
\begin{equation}
m_G(z)=\int \frac{1}{\lambda-z}dG(\lambda), \ \ \im z> 0
\end{equation}
and $G$ can be recovered by the inversion formula
$$
      G\{[x_1,x_2]\}=\frac{1}{\pi}\lim\limits_{\varepsilon\rightarrow
      0+}\int_{x_1}^{x_2} {\mathcal{I}}m (m_G(x+i\varepsilon)) dx,
$$
where $x_1$ and $x_2$ are continuity points of $G$.
\medskip

Moreover, equation (\ref{eq2}) takes a simpler form when $F^{c,H}$ is replaced by
\begin{equation}
\underline{F}^{c,H}=(1-c)I_{[0,\infty]}+cF^{c,H},
\end{equation}
which is the limiting ESD of $\underline{\textbf{A}}_n=\frac{1}{n}\bbX\bbX^{*}$. Its Stieltjes transform
\begin{equation}\label{stieljes}
\underline{m}(z)=-\frac{1-c}{z}+cm(z)
\end{equation}
has an inverse
\begin{equation}
z=z(\underline{m})=\frac{1}{\underline{m}}+c\int\frac{x}{1+x \underline{m}}dH(x).
\end{equation}

The construction of our test statistic relies on the following observation: the limit of the ESD of the sample covariance matrix $\bbA_n$ is the M-P law by (\ref{a2}) when $\bbx_1,\cdots,\bbx_p$ are independent and satisfy Assumption 1, while the limit of the ESD is determined from (\ref{eq2}) when there is some correlation among $\bbx_1,\cdots,\bbx_p$ with the covariance matrix $\bbT_p$ different from the identity matrix. Moreover, preliminary investigations indicate that when $\bbx_1,\cdots,\bbx_p$ are only uncorrelated (without any further assumptions), the limit of the ESD of $\bbA_n$ is not the M-P law (see \cite{RD2009}). These therefore motivate us to employ the ESD of $\bbA_n$, $F^{\bbA_n}(x)$, as a test statistic. There is no central limit theorem for $(F^{\bbA_n}(x)-F^{c,H}(x))$, however, as argued by \cite{BS2004}. We instead consider the characteristic function of $F^{\bbA_n}(x)$.
\medskip

The characteristic function of $F^{\bbA_n}(x)$ is
\begin{equation}
s_n(t)\triangleq\int e^{itx}dF^{\bbA_n}(x)=\frac{1}{p}\sum\limits_{i=1}^pe^{it\lambda_i},
\end{equation}
where $\lambda_i, i=1,\ldots,p$ are eigenvalues of the sample covariance matrix $\bbA_n$.
\medskip

Our test statistic is then proposed as follows:
\begin{equation}\label{a6}
M_n=\int^{T_2}_{T_1}|s_n(t)-s(t)|^2dU(t),
\end{equation}
where $s(t):=s(t,c_n)$ is the characteristic function of $F^{c_n}(x)$, obtained from the M-P law $F^{c}(x)$ with $c$ replaced by $c_n=p/n$, and $U(t)$ is a distributional function with its support on a compact interval, say $[T_1,T_2]$.

To develop the asymptotic distribution of $M_n$ under a local alternative, the following assumption is needed.

\begin{assu}\label{assu2}
Let $\textbf{T}_p$ be a $p\times p$ random Hermitian nonnegative definite matrix with a bounded spectral norm. Let $\bby_j^T=\bbw_j^T\bbT_p^{1/2}$, where $\bbT_p^{1/2}$ is the $p\times p$ Hermitian matrix that satisfies $(\bbT_p^{1/2})^{2}=\bbT_p$ and $\bbw_j^T=(W_{j1},\cdots,W_{jp}), j=1,\ldots,n$ are i.i.d random vectors, in which $W_{ji}, j\leq n, i\leq p$ are i.i.d with mean zero, variance one and finite fourth moment.

The empirical spectral distribution $F^{\textbf{T}_p}$ of $\textbf{T}_p$ converges weakly to a distribution $H$ on $[0,\infty)$ as $n\rightarrow \infty$;  all the diagonal elements of the matrix $\bbT_p$ are equal to $1$.
\end{assu}

Note that under Assumption \ref{assu2}, $\bbA_n$ becomes $\bbT_p^{1/2}\bbW^*\bbW\bbT_p^{1/2}$, where $\bbW=(\bbw_1,\cdots,\bbw_n)^T$. The assumption that all the diagonal elements of $\bbT_p$ are equal to $1$ is used to guarantee that $EX_{ji}^2=1$. Roughly speaking, this constraint is to ensure that the data are weakly stationary.  However, we would like to remark that such a constraint could be removed if we knew that $\bbT_p$ was not a diagonal matrix.
 Under Assumption \ref{assu2}, when $\bbT_p=\bbI_p$, the random vectors $\bbx_1,\ldots,\bbx_p$ are independent and when $\bbT_p\neq\bbI_p$, they are not independent. For convenience, we name this dependent structure as `linear dependent structure'.
\medskip

To develop the asymptotic distribution of the test statistic, we introduce
\begin{equation}
G_n(x)=p[F^{\textbf{A}_n}(x)-F^{c_n}(x)].
\end{equation}

Then, $p(s_n(t)-s(t))$ can be decomposed as a sum of the random part and the non-random part as follows:
\bea
&& p(s_n(t)-s(t)) = \int e^{itx}dG_n(x)
\nonumber \\
&&= \ \int e^{itx}d\left(p\left[F^{\textbf{A}_n}(x)-F^{c_n,H_n}(x)\right]\right) +\int e^{itx}d\left(p\left[F^{c_n,H_n}(x)-F^{c_n}(x)\right]\right),\label{a3}
\eea
where $F^{c_n,H_n}$ is obtained from $F^{c,H}$ with $c$ and $H$ replaced by $c_n=p/n$ and $H_n=F^{\bbT_p}$.
\medskip

\begin{thm}\label{thm2}
Let Assumptions \ref{assu1} and \ref{assu3} hold.

1) Suppose that Assumptions \ref{assu2}  and the following conditions are satisfied:
\begin{equation}\label{a5}
\frac{1}{n}\sum^{n}_{i=1}\bbe^{*}_i\textbf{T}^{1/2}_p(\underline{m}(z_1)\textbf{T}_p+\textbf{I})^{-1}\textbf{T}_p^{1/2}\bbe_i\bbe^{*}_i\textbf{T}_p^{1/2}(\underline{m}(z_2)\textbf{T}_p+\textbf{I})^{-1}\textbf{T}^{1/2}_p\bbe_i\rightarrow h_1(z_1,z_2)
\end{equation}
and
\begin{equation}\label{a4}
\frac{1}{n}\sum^{n}_{i=1}\bbe^{*}_i\textbf{T}^{1/2}_p(\underline{m}(z)\textbf{T}_p+\textbf{I})^{-1}\textbf{T}_p^{1/2}\bbe_i\bbe^{*}_i\textbf{T}_p^{1/2}(\underline{m}(z)\textbf{T}_p+\textbf{I})^{-2}\textbf{T}^{1/2}_p\bbe_i\rightarrow h_2(z),
\end{equation}
where $\bbe^{*}_i$ is the $n$-dimensional row vector with the $i$-th element being $1$ and others $0$.

Then, the scaled proposed test statistic $p^2 M_n$ converges in distribution to a random variable $R$ of the form:
\begin{equation}\label{R}
R=\int^{T_2}_{T_1}(|V(t)+\delta_1(t)|^{2}+| Z(t)+\delta_2(t)|^{2})dU(t),
\end{equation}
where $(V(t), Z(t))$ is a vector of Gaussian processes and $\delta_j(t),\ j=1,2$ are defined as
\begin{eqnarray}
\delta_1(t)=\lim_{n\rightarrow\infty}\int cos(tx)dp\big(F^{c_n,H_n}(x)-F^{c_n}(x)\big),
\end{eqnarray}
\begin{eqnarray}
\delta_2(t)=\lim_{n\rightarrow\infty}\int sin(tx)dp\big(F^{c_n,H_n}(x)-F^{c_n}(x)\big).
\end{eqnarray}

The mean and variance of $(V(t), Z(t))$ are specified as follows.
If $X_{11}$ is real, then
\begin{eqnarray*}
EV(t)&=&\frac{1}{2\pi i}\oint_{\gamma} cos(tz)\frac{c\int \underline{m}^3(z)\tau^2(1+\tau\underline{m}(z))^{-3}dH(\tau)}{(1-c\int \underline{m}^2(z)\tau^2(1+\tau\underline{m}(z))^{-2}dH(\tau))^2}dz\\
&&-\frac{EX_{11}^4-3}{2\pi i}\oint_{\gamma} cos(t_jz)\frac{c\underline{m}^3(z)h_2(z)}{1-c\int \underline{m}^2(z)\tau^2(1+\tau\underline{m}(z))^{-2}dH(\tau)}dz.
\end{eqnarray*}

Replacing $cos(tz)$ in $EV(t)$ by $sin(tjz)$ yields the expression of $EZ(t)$.
Also
\begin{eqnarray*}
&& Cov\Big(V(t_j),Z(t_h)\Big) = -\frac{1}{2\pi^2}\oint_{\gamma_1}\oint_{\gamma_2} \frac{cos(t_jz_1)sin(t_hz_2)}{(\underline{m}(z_1)-\underline{m}(z_2))^2}\frac{d}{dz_2}\underline{m}(z_2)\frac{d}{dz_1}\underline{m}(z_1)dz_1dz_2\\
&&-\frac{c(EX_{11}^4-3)}{4\pi^2}\oint_{\gamma_1}\oint_{\gamma_2} cos(t_jz_1)sin(t_hz_2)\frac{d^2}{dz_1dz_2}[\underline{m}(z_1)\underline{m}(z_2)h_1(z_1,z_2)]dz_1dz_2.
\end{eqnarray*}

If $X_{11}$ is complex with $EX_{11}^2=0$, then
\begin{eqnarray*}
EV(t)=-\frac{E|X_{11}|^4-2}{2\pi i}\oint_{\gamma} cos(t z)\frac{c\underline{m}^3(z)h_2(z)}{1-c\int \underline{m}^2(z)\tau^2dH(\tau)/(1+\tau\underline{m}(z))^2}dz,
\end{eqnarray*}
and the covariance
\begin{eqnarray*}
&& Cov\Big(V(t_j),Z(t_h)\Big) =-\frac{1}{4\pi^2}\oint_{\gamma_1}\oint_{\gamma_2}
\frac{cos(t_jz_1)sin(t_hz_2)}{(\underline{m}(z_1)-\underline{m}(z_2))^2}\frac{d}{dz_2}\underline{m}(z_2)\frac{d}{dz_1}\underline{m}(z_1)dz_1dz_2\\
&&-\frac{c(E|X_{11}|^4-2)}{4\pi^2}\oint_{\gamma_1}\oint_{\gamma_2} cos(t_jz_1)sin(t_hz_2)\frac{d^2}{dz_1dz_2}[\underline{m}(z_1)\underline{m}(z_2)h_1(z_1,z_2)]dz_1dz_2.
\end{eqnarray*}

The covariances $Cov(V_j,V_h)$ and $Cov(Z_j, Z_h)$ are similar to $Cov(V_j, Z_h)$ except replacing $(cos(t_jz),sin(t_hz))$ by $(cos(t_jz), cos(t_hz))$ and $(sin(t_jz), sin(t_hz))$ respectively.

The contours $\gamma$, $\gamma_1$ and $\gamma_2$ above are all closed and are taken in the positive direction in the complex plane, each enclosing the support of $F^{c,H}$. Also $\gamma_1$ and $\gamma_2$ are disjoint.

2) Under the null hypothesis $\bbH_0$, the scaled statistic $p^2M_n$ then converges in distribution to
\begin{eqnarray}\label{R0}
R_0=\int^{T_2}_{T_1}(|\tilde V(t)|^{2}+|\tilde Z(t)|^{2})dU(t),
\end{eqnarray}
where the distribution of $(\tilde V(t), \tilde Z(t))$ can be obtained from $( V(t), Z(t))$ with $H(\tau)$ becoming the degenerate distribution at the point $1$, $\underline{m}(z)$ being the
Stieltjes transform of the M-P law, $h_1(z_1,z_2)=\frac{1}{(\underline{m}(z_1)+1)(\underline{m}(z_2)+1)}$ and $h_2(z)=\frac{1}{(\underline{m}(z)+1)^3}$. 

\end{thm}
\medskip

\begin{rmk}
Assumption \ref{assu1} assumes that all the entries of $\bbx_i$ are identically distributed. It is of practical interest to consider removing the identical distribution condition. Instead of assuming identically distributed entries for $\bbx_i$,
we need only to impose the following additional assumptions: \ for any $k=1,\ldots,p; j=1,\ldots,n$, $EX_{jk}=0$, $EX_{jk}^2=1$, $\sup\limits_{j,k}EX_{jk}^{4}<\infty$ and for any $\eta>0$,
\be
\frac{1}{\eta^4np}\sum_{j=1}^n \sum_{k=1}^{p} E(|X_{jk}|^4I_{(|X_{jk}|\geq \eta \sqrt{n})})\rightarrow 0.
\label{a9}
\ee

A careful checking on the arguments of Theorem $1.1$ of \cite{BS2004} and Theorem $1.4$ of \cite{PZ2008} indicates that Lemma 6 (of Appendix A of the supplementary material) still holds and hence Theorem \ref{thm2} holds under (\ref{a9}). For the expressions of the mean and covariance of the asymptotic random vector, we substitute the fourth moment $E|W_{11}|^4$ in Lemma 6 with the average of all the fourth moments of all the entries, i.e. $\frac{\sum^{n,p}_{j=1,k=1}E|W_{jk}|^4}{np}$.
\end{rmk}
\medskip

Conditions (\ref{a5}) and (\ref{a4}) can be removed if $E\left[W_{11}^4\right]=3$ in the real--number case or $E\left[W_{11}^2\right]=2$ in the complex--number case (see, for example, \cite{BS2004}). The second part of the above theorem is concerned with asymptotic distributions of the test statistic under a local alternative hypothesis, i.e., Assumption \ref{assu2}. With respect to Assumption \ref{assu2}, we would make the following comments, which are useful in the subsequent application section.

If $\bby_j^{T}=\tilde{\bbw}_j^{T}\bbC$, where $\bbC$ is any $q\times p$ nonrandom matrix and $\tilde{\bbw}_j, j=1,\ldots,n$ are i.i.d. $q\times 1$ random vectors with their respective entries being i.i.d random variables, then Theorem \ref{thm2} is still applicable. This is because $\frac{1}{n}\bbX\bbX^{*}$ in this case becomes $\frac{1}{n}\tilde{\bbW}\bbC\bbC^{*}\tilde{\bbW}^{*}$ and the nonnegative definitive matrix $\bbC\bbC^{*}$ can be decomposed into $\bbC\bbC^{*}=\bbT_q^{1/2}\bbT_q^{1/2}$, where $\bbT_q^{1/2}$ is a $q\times q$ Hermitian matrix and $\tilde{\bbW}=(\tilde{\bbw}_1,\ldots,\tilde{\bbw}_n)$. Note that the eigenvalues of $\frac{1}{n}\bbX\bbX^{*}$ differ from those of $\frac{1}{n}\bbT_q^{1/2}\bbW^{*}\bbW\bbT_q^{1/2}$ by $|p-q|$ zeros. Thus, we may instead resort to CLT of $\frac{1}{n}\bbT_q^{1/2}\tilde{\bbW}^{*}\tilde{\bbW}\bbT_q^{1/2}$.

\section{Applications in multiple MA(1), AR(1) and spatial cross--sectional dependence structures}

This subsection is to explore some applications of the proposed test.
\medskip

\noindent{\bf Example 3.1}. \
Consider a multiple moving average model of order $1$(MA(1)) of the form:
\begin{equation}\label{eqma1}
\bbv_t=\bbz_t+\psi\bbz_{t-1}, t=1,\ldots,p,
\end{equation}
where $|\psi|<1$; $\bbz_t=(Z_{1t},\ldots,Z_{nt})^T$ is an $n$-dimensional random vector of i.i.d. elements, each of which has zero mean and unit variance; and $\bbv_t=(V_{1t},\ldots,V_{nt})^T$. Denote by $\hat \bbV_j^{T}$ and $\hat \bbZ_j^{T}$ respectively the $j$-th rows of $\bbV=(V_{jt})_{n\times p}$ and $\bbZ=(Z_{jt})_{n\times (p+1)}$.

For each $j=1,\ldots,n$, the MA(1) model (\ref{eqma1}) can be written as
\begin{eqnarray}
\hat\bbV_j^{T}=\hat\bbZ_j^{T}\bbC,
\end{eqnarray}
where
\begin{eqnarray}
\bbC=\left(
       \begin{array}{cccccc}
         \psi & 0 & 0   & \cdots & 0 & 0 \\
         1 & \psi   & 0 & \cdots & 0 & 0 \\
         0 & 1   & \psi   & \cdots & 0 & 0 \\
         \vdots & \vdots & \vdots & \ddots & \vdots & \vdots \\
         0 & 0   & 0   & \cdots & 1 & \psi \\
         0 & 0   & 0   & \cdots & 0 & 1 \\
       \end{array}
     \right)_{(p+1)\times p}.
\end{eqnarray}

From Assumption \ref{assu2}, the last paragraph of the preceding subsection and Theorem \ref{thm2}, our test is able to capture the dependence of $\bbv_1,\cdots,\bbv_p$ as $n$ and $p$ go to infinity in the same order.
\medskip

\noindent{\bf Example 3.2}. \
Consider a multiple autoregressive model of order $1$(AR(1)) of the form:
\begin{eqnarray}\label{eqar1}
\bbv_t=\phi\bbv_{t-1}+\bbz_t, \ t=1,\ldots,p, \ \bbv_0=\frac{1}{\sqrt{1-\phi^2}}\bbz_0,
\end{eqnarray}
where $|\phi|<1$; for any $t=0,1,\ldots,p$, $\bbz_t=(Z_{1t},\ldots,Z_{nt})^T$ is an $n$-dimensional random vector with i.i.d. elements, each of which has zero mean and unit variance; and $\bbv_t=(V_{1t},\ldots,V_{nt})^T$. Denote the $j$-th rows of $\bbV=(V_{jt})_{n\times (p+1)}$ and $\bbZ=(Z_{jt})_{n\times (p+1)}$ as $\hat \bbV_j^{T}$ and $\hat \bbZ_j^{T}$ respectively.

For each $j=1,\ldots,n$, the AR(1) model (\ref{eqar1}) can be written as
\begin{eqnarray}
\hat\bbV_j^{T}=\hat\bbZ_j^{T}\bbD,
\end{eqnarray}
where
\begin{eqnarray}
\bbD=\left(
       \begin{array}{cccccc}
         1 & -\phi & (-\phi)^2 & \cdots & (-\phi)^{p-2} & (-\phi)^{p-1}/\sqrt{1-\phi^2} \\
         0 & 1 & -\phi & \cdots & (-\phi)^{p-3} & (-\phi)^{p-2}/\sqrt{1-\phi^2} \\
         0 & 0 & 1 & \cdots & (-\phi)^{p-4} & (-\phi)^{p-3}/\sqrt{1-\phi^2} \\
         \vdots & \vdots & \vdots & \ddots & \vdots & \vdots \\
         0 & 0 & 0 & \cdots & 1 & -\phi/\sqrt{1-\phi^2} \\
         0 & 0 & 0 & \cdots & 0 & 1/\sqrt{1-\phi^2} \\
       \end{array}
     \right).
\end{eqnarray}

By Theorem \ref{thm2}, we can also apply the proposed test $M_n$ to this AR(1) model as well.
\medskip

\noindent{\bf Example 3.3}. \
We now consider a panel data case. Let $\{v_{ji}: i=1,\ldots,p; j=1,\ldots,n\}$ be the error components in a panel data model. They may be cross--sectionally correlated. In panel data analysis, it is of interest to consider the cross--sectional independence hypothesis, i.e.
\begin{equation*}
\bbH_{00}: \ Cov(v_{ji},v_{jh})=0 \ for\ all\ j=1,\ldots,n\ and\ all\ i\neq h;
\end{equation*}
against
\begin{equation}\label{b4}
\bbH_{11}: \ Cov(v_{ji},v_{jh})\neq 0 \ for\ some\ j\ and\ some \ i\neq h.
\end{equation}
Under the assumption that $\{v_{ji}: i=1,\ldots,p; j=1,\ldots,n\}$ are normal distributed, this hypothesis is equivalent to the independence hypothesis that
\begin{eqnarray}\label{b10}
\bbH_0:\ \bbv_1,\ldots,\bbv_p\ are \ independent; \ \ against\ \ \bbH_1:\ \bbv_1,\ldots,\bbv_p\ are \ not \ independent,
\end{eqnarray}
where $\bbv_i=(v_{1i},\ldots,v_{ni})^{T}$, $i=1,\ldots,p$.

Modern panel data literature has mainly adopted two different approaches to model error cross--sectional dependence: the spatial approach and the factor-structure approach. For the spatial approach, there are three popular spatial models: Spatial Moving Average (SMA), Spatial Auto-Regressive (SAR) and Spatial Error Components (SEC) processes. They are defined respectively as follows:
\begin{equation}\label{b1}
SMA:\ v_{ji}=\sum^{p}_{k=1}\omega_{ik}\varepsilon_{jk}+\varepsilon_{ji},
\end{equation}
\begin{equation}\label{b2}
SAR:\ v_{ji}=\sum^{p}_{k=1}\omega_{ik}v_{jk}+\varepsilon_{ji},
\end{equation}
\begin{equation}\label{b3}
SEC:\ v_{ji}=\sum^{p}_{k=1}\omega_{ik}\xi_{jk}+\varepsilon_{ji},
\end{equation}
where $\omega_{ik}$ is the $i$--specific spatial weight attached to individual $k$; $\{\varepsilon_{ji}:i=1,\ldots,p; j=1,\ldots,n\}$ and $\{\xi_{ji}: i=1,\ldots,p; j=1,\ldots,n\}$ are two sets of i.i.d. random components with zero mean and unit variance, and $\{\xi_{ji}: i=1,\ldots,p; j=1,\ldots,n\}$ are uncorrelated with $\{\varepsilon_{ji}, i=1,\ldots,p; j=1,\ldots,n\}$.

 Denote the $j$-th row of $\bbV=(v_{ji})_{n\times p}$, $\boldsymbol\varepsilon=(\varepsilon_{ji})_{n\times p}$ and
$\boldsymbol\xi=(\xi_{ji})_{n\times p}$ by $\hat\bbv_j^{T}$, $\hat{\boldsymbol\varepsilon}_j^{T}$ and $\hat{\boldsymbol\xi}_j^{T}$ respectively. Set $\boldsymbol\omega=(\omega_{ik})_{p\times p}$.
Model SMA (\ref{b1}) may be rewritten as $\hat\bbv_j^{T}=\hat{\boldsymbol\varepsilon}_{j}^{T}(\boldsymbol\omega^{T}+\bbI_p), \forall j=1,\ldots,n$ and hence $\bbT_p=(\boldsymbol\omega+\bbI_p)(\boldsymbol\omega^{T}+\bbI_p)$.
For model SAR (\ref{b2}), assume that $\boldsymbol{\omega}-\bbI_p$ is invertible. 
 We then write $\hat\bbv_j^{T}=\hat{\boldsymbol\varepsilon}_{j}^{T}(\boldsymbol\omega^{T}-\bbI_p)^{-1}, \forall j=1,\ldots,n$. Hence $\bbT_p=(\boldsymbol\omega-\bbI_p)^{-1}(\boldsymbol\omega^{T}-\bbI_p)^{-1}$. Therefore the test statistic $M_n$ can be used to identify whether $\bbv_1,\cdots,\bbv_p$ of models (\ref{b1}) and (\ref{b2}) are independent. Hence it can capture the cross--sectional dependence for the SMA model and SAR model

As for the SEC model defined in (\ref{b3}), whether the statistic $M_n$ can detect the dependence structure of the SEC model relies on the properties of a sample covariance matrix of the form:
\begin{equation}
\bbB_n=\frac{1}{n}(\boldsymbol\omega\boldsymbol\xi+\boldsymbol\varepsilon)(\boldsymbol\omega\boldsymbol\xi+\boldsymbol\varepsilon)^{T},
\end{equation}
where $\boldsymbol\xi=(\boldsymbol\xi_1,\ldots,\boldsymbol\xi_p)^{T}$ and $\boldsymbol\varepsilon=(\boldsymbol\varepsilon_1,\ldots,\boldsymbol\varepsilon_p)^{T}$.

Under the null hypothesis $\bbH_0$, \cite{DS2007} provides the LSD of the matrix $\bbB_n$ whose Stieljes transform is
\begin{equation}
\hat m(z)=\int\frac{d\hat H(x)}{\frac{x}{1+c\hat m(z)}-(1+c\hat m(z))z+1-c},
\end{equation}
where $\hat H(x)$ is the limit of $F^{\frac{1}{n}\boldsymbol\varepsilon\boldsymbol\varepsilon^{T}}$.
\medskip

With this result, we know that the LSD of the matrix $\bbB_n$ is not the M-P law. In view of this, the proposed test $M_n$ should capture the dependence of the SEC model (\ref{b3}) in theory. However, the asymptotic distribution of the proposed test statistic $M_n$ for this case will need to be developed in future work.

\section{A general panel data model}

Note that the proposed test is based on the idea that the limits of ESDs under the null and local alternative hypotheses are different. Yet, it may be the case where there exists some dependence among the set of vectors of $\bbv_1,\cdots,\bbv_p$ but the limit of the ESD associated with such vectors is the M-P law.
Then a natural question is whether the statistic $M_n$ works in this case. We below investigate a panel data model as an example.

Consider a panel data model of the form
\begin{eqnarray}\label{panel1}
v_{ij}=\varepsilon_{ij}+\frac{1}{\sqrt{p}}u_i,\ i=1,\ldots,p; \ j=1,\ldots,n,
\end{eqnarray}
where $\{\varepsilon_{ij}, i=1,\ldots,p; j=1,\ldots,n\}$ is a sequence of i.i.d. real random variables with $E\varepsilon_{11}=0$ and $E\varepsilon_{11}^2=1$, and $\{u_i,i=1,\ldots,p\}$ are real random variables, and independent of $\{\varepsilon_{ij}, i=1,\ldots,p; j=1,\ldots,n\}$.

For any $i=1,\ldots,p$, set
\begin{eqnarray}
\bbv_i=(v_{i1},\ldots,v_{in})^{T}.
\end{eqnarray}

The aim of this section is to test the null hypothesis specified in (\ref{b10}) for model (\ref{panel1}).

Model (\ref{panel1}) can be written as
\begin{equation}
\bbv=\boldsymbol\varepsilon+\bbu\bbe^{T},
\end{equation}
where $\bbv=(\bbv_1,\ldots,\bbv_p)^{T}$, $\bbu=(\frac{1}{\sqrt{p}}u_1,\ldots,\frac{1}{\sqrt{p}}u_p)^{T}$ and $\bbe$ is $p\times 1$ vector with all elements being one.

Consider the sample covariance matrix
\begin{equation}\label{h2}
\bbS_n=\frac{1}{n}\bbv\bbv^{T}=\frac{1}{n}(\boldsymbol\varepsilon+\bbu\bbe^{T})(\boldsymbol\varepsilon+\bbu\bbe^{T})^{T}.
\end{equation}

By Lemma \ref{lem6} in the appendix and the fact that $rank(\bbu\bbe^{T})\leq 1$, it can be concluded that the limit of the ESD of the matrix $\bbS$ is the same as that of the matrix $\frac{1}{n}\boldsymbol\varepsilon\boldsymbol\varepsilon^{T}$, i.e. the M-P law.
Even so, we still would like to use the proposed statistic $M_n$ to test the null hypothesis of mutual independence. However, this model does not necessarily satisfy Assumption \ref{assu1} because the elements of each vector $\bbv_i$ are not independent and they include the common random factor $u_i$. As a consequence, Theorem \ref{thm2} thus can not be directly applied to this model. Therefore, we need to develop a new asymptotic theory for the proposed statistic $M_n$ for this model.


\begin{thm}\label{thm4}
For model (\ref{panel1}), in addition to Assumptions \ref{assu1} and \ref{assu3}, we assume that
\begin{equation}\label{asu}
E\|\bbu\|^4<\infty \ \ \mbox{and} \ \ \frac{1}{p^2}E\left[\sum^{p}_{i\neq j}(u_i^2-\bar u)(u_j^2-\bar u)\right]\rightarrow 0 \ \ as\ n\rightarrow\infty,
\end{equation}
where $\bar u$ is a positive constant.

Then, the proposed test statistic $p^2M_n$ converges in distribution to the random variable $R_2$ given by
\begin{eqnarray}
R_2=\int^{t_2}_{t_1}\big(|W(t)|^2+|Q(t)|^2\big)dU(t),
\end{eqnarray}
where $(W(t),Q(t))$ is a Gaussian vector whose mean and covariance are specified as follows:
\begin{eqnarray}\label{panel12}
E\left[W(t)\right] &=&-\frac{c}{2\pi i}\oint_{\gamma} cos(t_jz)\frac{\underline{m}^3(z)(1+\underline{m}(z))}{((1+\underline{m}(z))^2-c\underline{m}^2(z))^2}dz
\nonumber\\
&& -\frac{c(EY_{11}^4-3)}{2\pi i}\oint_{\gamma}cos(t_jz)\frac{\underline{m}^3(z)}{(1+\underline{m}(z))^2-c\underline{m}^2(z)}dz\non
&&+\frac{c}{2\pi i}\oint_{\gamma}cos(t_jz)\frac{\underline{m}(z)}{z((1+\underline{m}(z))^2-c\underline{m}^2(z))}dz\non
&&-\frac{1}{2\pi i}\oint_{\gamma}cos(t_jz)\frac{\frac{c\underline{m}(z)}{z[(1+\underline{m}(z))^2-c\underline{m}^2(z)]}+\bar u\int\frac{1}{(\lambda-z)^2}dF^{MP}(\lambda)}
{\bar uz\underline{m}^2(z)-1}dz
\end{eqnarray}
and
\begin{eqnarray}\label{panel13}
&&Cov(W(t_j),Q(t_h))=-\frac{1}{2\pi^2}\oint_{\gamma_1}\oint_{\gamma_2}\frac{cos(t_jz_1)sin(t_hz_2)}{(\underline{m}(z_1)-\underline{m}(z_2))^2}
\frac{d}{dz_1}\underline{m}(z_1)\frac{d}{dz_2}\underline{m}(z_2)dz_1dz_2\non
&&-\frac{c(EY_{11}^4-3)}{4\pi^2}\oint_{\gamma_1}\oint_{\gamma_2}cos(t_jz_1)cos(t_hz_2)\frac{d^2}{dz_1dz_2}
[\frac{\underline{m}(z_1)\underline{m}(z_2)}{(1+\underline{m}(z_1))(1+\underline{m}(z_2))}]dz_1dz_2.
\end{eqnarray}
Replacing $cos(t_jz)$ in $E\left[W(t_j)\right]$ by $sin(t_jz)$ yields the expression of $E\left[Q(t_j)\right]$. The covariances $Cov(W(t_j),W(t_h))$ and $Cov(Q(t_j),Q(t_h))$ are similar except replacing $sin(t_hz)$ and $cos(t_jz)$ by $cos(t_hz)$ and $sin(t_jz)$ respectively. The contours in (\ref{panel12}) and (\ref{panel13}) both enclose the interval $[(1-\sqrt{c})^2+2c\bar u, (1+\sqrt{c})^2+2c\bar u]$. Moreover, the contours $\gamma_1$ and $\gamma_2$ are disjoint.
\end{thm}

\begin{rmk}
When $u_1,\cdots,u_p$ are independent and hence $\bbv_1,\cdots,\bbv_p$ are independent, condition (\ref{asu}) is true. 
\end{rmk}

In view of Theorem \ref{thm4}, we see that the proposed test statistic $M_n$ still works mainly due to the involvement of the last term on the right hand of (\ref{a3}).

\section{Small sample simulation studies}

This section provides some simulated examples to show the finite sample performance of the proposed test. In addition, we also compare the performance of the proposed test with that of a likelihood ratio test proposed by \cite{Ader1984}. Simulations are used to compute and find the empirical sizes and powers of the proposed test and hence evaluate the performance of the test. To show the efficiency of our test, the two kinds of dependence structures investigated in Section $3$ and $5$ are detected, such as multiple MA(1) and AR(1) model, SMA and the general panel data model.

\subsection{Empirical sizes and power values}

First we introduce the method of calculating empirical sizes and empirical powers. Let $z_{\frac{1}{2}\alpha}$ and $z_{1-\frac{1}{2}\alpha}$ be the $100(\frac{1}{2}\alpha)\%$ and $100(1-\frac{1}{2}\alpha)\%$ quantiles of the asymptotic null distribution of the test statistic $M_n$ respectively. With $K$ replications of the data set simulated under the null hypothesis, we calculate the empirical size as
\begin{equation}
\hat \alpha=\frac{\{\sharp \ of \ M_n^H\geq z_{1-\frac{1}{2}\alpha}\ or\ M_n^H\leq z_{\frac{1}{2}\alpha}\}}{K},
\end{equation}
where $M_n^H$ represents the values of the test statistic $M_n$ based on the data simulated under the null hypothesis.

In our simulation, we choose $K=1000$ as the number of repeated simulations. The significance level is $\alpha=0.05$. Since the asymptotic null distribution of the test statistic is not a classical distribution, we need to estimate the quantiles $z_{\frac{1}{2}\alpha}$ and $z_{1-\frac{1}{2}\alpha}$. Naturally, we do as follows: generate $K$ replications of the asymptotic distributed random variable and then select the $(K\frac{1}{2}\alpha)$-th smallest value $\hat z_{\frac{1}{2}\alpha}$ and $(K\frac{1}{2}\alpha)$-th largest value $\hat z_{1-\frac{1}{2}\alpha}$ as the estimated $100(\frac{1}{2}\alpha)\%$ and $100(1-\frac{1}{2}\alpha)\%$ quantiles of the asymptotic distributed random variable.

With the estimated critical points $\hat z_{\frac{1}{2}\alpha}$ and $\hat z_{1-\frac{1}{2}\alpha}$ under the null hypothesis, the empirical power is calculated as
\begin{equation}
\hat \beta=\frac{\{\sharp \ of \ M_n^A\geq \hat z_{1-\frac{1}{2}\alpha}\ or\ M_n^A\leq \hat z_{\frac{1}{2}\alpha}\}}{K},
\end{equation}
where $M_n^A$ represents the values of the test statistic $M_n$ based on the data simulated under the alternative hypothesis.

\subsection{Comparisons with the classical likelihood ratio test}
For the proposed independence test, we generate $n$ numbers of $p$-dimensional independent and identical distributed random vectors $\{\bby_{j}\}_{j=1}^{n}$, each with the mean vector $\textbf{0}_p$ and the covariance matrix $\Sigma$. Under
 the null hypothesis, $\{\bby_j\}_{j=1}^n$ are generated in two scenarios:
\begin{enumerate}
\item Each $\bbw_j$ is a $p$-dimensional normal random vector with the mean vector $\textbf{0}_p$ and the covariance matrix $\Sigma=\textbf{I}_p$ for $j=1,\ldots n$, $\bby_j=\bbT_p\bbw_j$ with $\textbf{T}_p=\textbf{I}_p$;
\item Each $\bbw_j$ consists of i.i.d. random variables with standardized Gamma(4,2) distribution, so they have zero means and unit variances; for $j=1,\ldots n$, $\bby_j=\bbT_p\bbw_j$ with $\textbf{T}_p=\textbf{I}_p$.
\end{enumerate}

Under the alternative hypothesis, we consider the case: $\bbT_p^{1/2}=(\sqrt{0.95}\textbf{I}_p, \sqrt{0.05}\textbf{1}_p)$, where $\textbf{1}_p$ is a $p$--dimensional vector with $1$ as entries. In this case, the population covariance matrix of $\bby_j$ is $\Sigma=0.95\textbf{I}_p+0.05\textbf{1}_p\textbf{1}_p^{'}$, which is called the compound symmetric covariance matrix.

For normal distributed data, the fourth moment of each element is $E|X_{11}|^4=3$; for standardized Gamma(4,2) distributional data, $E|X_{11}|^4=4.5$. This can be calculated by a formula as follows: the $k$-th moment of $X_{11}$ which is $Gamma(\alpha, \lambda)$ is
\begin{equation}
E|X_{11}|^k=\frac{(\alpha+k-1)\cdots (\alpha+1)\alpha}{\lambda^k}.
\end{equation}

\cite{Ader1984} provides a likelihood ratio criterion (LRT) to test independence for a fixed number of fixed dimensional normal distributed random vectors. We compare it with the proposed test.

Under the null hypothesis, the distribution of $L$ is the distribution of $L_2L_3\cdots L_p$, where $L_2,\ldots,L_p$ are independently distributed with $L_k$ having the distribution of $U_{m,(k-1)m,n-1-(k-1)m}$. Furthermore, for any $k=2,\ldots,p$, as $n\rightarrow\infty$, $-(n-\frac{3}{2}-\frac{km}{2})log\left[U_{m,(k-1)m,n-1-(k-1)m}\right]$ has $\chi^2$ distribution with $(k-1)m^2$ degrees of freedom (see Section $8.5$ of \cite{Ader1984}).

From the construction of the LRT test, we can see that the LRT utilizes additional $n$ observations of the random vectors $\bbx_1,\ldots,\bbx_p$ under investigation, while the proposed test does not need this information. However, we can choose $m=1$ and apply LRT for testing the independence of the random vectors $\tilde{\bbx}_1,\ldots,\tilde{\bbx}_p$, where for any $i=1,\ldots,p$, the elements of the vector $\tilde{\bbx}_i$ consist of its $n$ observations. Hence the LRT test can test independence for the number of $p$ random vectors with dimension $n$ by choosing $m=1$.

Tables \ref{tb1} and \ref{tb2} show the empirical sizes and power values of our proposed test and the LRT test for normal distributed random vectors respectively. From Tables \ref{tb1} and \ref{tb2}, we can see that the LRT test does not work when $p$ and $n$ are both large while the proposed test possesses good performance when $p$ and $n$ go to infinity at the same order. The LRT test is only applicable to the case where $p$ is fixed and $n$ tends to infinity. From Table \ref{tb2}, it can be seen that the LRT fails when $p$ is comparable with $n$. When the difference between $p$ and $n$ are large, the sizes and the powers of the proposed test become worse. This is because our test is proposed under the case that $p$ and $n$ are in the same order when they approximate to infinity. The proposed test  also works well for gamma random vectors while the LRT test is not applicable to gamma case, since, in theory, LRT test is provided for normal random vectors. Table \ref{tb3} provides the empirical sizes and empirical powers of the proposed test for the gamma case. In our simulation, we choose $p,n=5, 10, 20, 30, 40, 50, 60, 70, 80, 90,100$ for the proposed test and the LRT test. The significant level $\alpha$ is chosen as $0.05$. In each case, we run $K=1000$ repeated simulations. Reports for empirical powers show that the proposed test can check independence for both normal and gamma vectors well. Moreover, the empirical powers converge to $1$ as $n,p\rightarrow \infty$.

\subsection{Multiple MA(1), AR(1) and SMA model}

Consider multiple MA(1) model \begin{equation}\label{ma1}
\bbv_t=\bbz_t+\psi\bbz_{t-1}, t=1,\ldots,p.
\end{equation} We choose $\psi=0.5$ and the simulation results in Table \ref{tb4} show that the proposed test performs well for this model.

Consider multiple AR(1) \begin{eqnarray}\label{ar1}
\bbv_t=\phi\bbv_{t-1}+\bbz_t, \ t=1,\ldots,p, \ \bbv_0=\frac{1}{\sqrt{1-\phi^2}}\bbz_0.
\end{eqnarray}
Let $\phi=0.5$. The empirical powers for this model are provided in Table \ref{tb6AR}. As $n$ and $p$ increase in the same order, the empirical power tends to $1$.

As for the Spatial Moving Average (SMA) model, i.e.
\begin{equation}
\ v_{ji}=\sum^{p}_{k=1}\omega_{ik}\varepsilon_{jk}+\varepsilon_{ji},
\end{equation}
 generate $\varepsilon_{jk}\overset{i.i.d}{\thicksim}normal(1,1), \forall j=1,\ldots,n; k=1,\ldots,p.$

Applying the proposed statistic $M_n$ for the sample matrix $\frac{1}{n}\bbV^{*}\bbV$, the empirical power values given in Table \ref{tb6SMA} show that $M_n$ performs well for capturing the cross-sectional dependence for SMA model.

\subsection{The general panel data model}

We examine the finite sample performance of the proposed test for the general panel data model (\ref{panel1}), i.e.
\begin{eqnarray}\label{panel1*}
v_{ij}=\varepsilon_{ij}+\frac{1}{\sqrt{p}}u_i,\ i=1,\ldots,p; \ j=1,\ldots,n,
\end{eqnarray}
where $\{\varepsilon_{ij}, i=1,\ldots,p; j=1,\ldots,n\}$ is a sequence of i.i.d. random variables with $E\varepsilon_{11}=0$ and $E\varepsilon_{11}^2=1$, and $\{u_i,i=1,\ldots,p\}$ are independent of $\{\varepsilon_{ij}, i=1,\ldots,p; j=1,\ldots,n\}$.

Under the null hypothesis, we generate $u_i\overset{i.i.d}{\thicksim}normal(1,1)$, $i=1,\ldots,p$ and under the alternative hypothesis, we generate the data from $\bbu=(\frac{1}{\sqrt{p}}u_1,\frac{1}{\sqrt{p}}u_2,\ldots,\frac{1}{\sqrt{p}}u_p)\thicksim \frac{1}{\sqrt{p}}N(\textbf{1}_{p}, \Sigma)$, where $\Sigma=\bbT\bbT^{T}$ and $\bbT$ is a $p\times p$ matrix with its elements being generated by $t_{ik}\overset{i.i.d.}{\thicksim}U(0,1)$, $i,k=1,\ldots,p$.

The simulation results for the empirical sizes and power values given in Table \ref{tb8} show that the proposed test can capture the dependence for the general panel data model (\ref{panel1}).

\subsection{Some other time series models and Vandermonde matrix}

Dependent structures of a set of random vectors are normally described by non-zero correlations among them, such as the linear dependent structure developed in Section $3$. However, there are some data that are not independent but uncorrelated. We consider three such examples and test their dependence by the proposed test although we have not developed an asymptotic theory for each of such local alternatives.

\subsubsection{Nonlinear MA model}

Consider a nonlinear MA model of the form
\begin{equation}\label{nMA}
R_{tj}=Z_{t-1,j}Z_{t-2,j}(Z_{t-2,j}+Z_{tj}+1),\ t=1,\ldots,p;\ j=1,\ldots,n;
\end{equation}
where $\bbz_t=(Z_{t1},\ldots,Z_{tn})$ is an $n$-dimensional random vector with i.i.d. elements, each of which has zero mean and unit variance, and $\bbr_t=(R_{t1},\ldots,R_{tn})$.

For any $j=1,\ldots,n$, the correlation matrix of $(R_{j1}, R_{j2}, \ldots, R_{jp})$ is a diagonal matrix. This model has been discussed by \cite{KuanLee} for testing a martingale difference hypothesis. Our proposed independence test can be applied to this nonlinear MA model, and the power values listed in Table \ref{tb5} show that the proposed test performs well for this model.

This discussion also indicates that the limit of the empirical spectral distribution of the nonlinear MA model (\ref{nMA}) is not the M-P law since the proposed test statistic is established on the characteristic function of the M-P law.

\subsubsection{Multiple ARCH(1) model}

Consider a multiple autoregressive conditional heteroscedastic (ARCH(1)) model of the form:
\begin{equation}
W_{tj}=Z_{tj}\sqrt{\alpha_0+\alpha_1W_{t-1,j}^2}, \ t=1,\ldots,p;\ j=1,\ldots, n;
\end{equation}
where $\bbz_t=(Z_{t1},\ldots,Z_{tn})$ is an $n$-dimensional random vector with i.i.d. elements, each of which has zero mean and unit variance, and $\boldsymbol\omega_t=(W_{t1},\ldots,W_{tn})$.

For each $j=1,\ldots,n$, ARCH(1) model $(W_{1j}, W_{2j}, \ldots, W_{pj})$ is a martingale difference sequence. ARCH(1) model has many applications in financial analysis. There exists no theoretical results stating that the LSD of the sample covariance matrix for ARCH(1) model is M-P Law, but from Figure \ref{fig0}, we can see that the LSD of the sample covariance matrix for ARCH(1) model is indeed M-P Law. A rigorous study is under investigation.  For the ARCH(1) model, the proposed test can not capture the dependence of $(\boldsymbol\omega_1,\boldsymbol\omega_2,\ldots,\boldsymbol\omega_p)$ directly, but we can test the dependence of
$(\boldsymbol\omega_1^2,\boldsymbol\omega_2^2,\ldots,\boldsymbol\omega_p^2)$. Since this test can tell us that $(\boldsymbol\omega_1^2,\boldsymbol\omega_2^2,\ldots,\boldsymbol\omega_p^2)$ are not independent, naturally it can be concluded that $(\boldsymbol\omega_1,\boldsymbol\omega_2,\ldots,\boldsymbol\omega_p)$ are not independent either. Here, we take $\alpha_0=0.9$ and $\alpha_1=0.1$. Table \ref{tb6} shows the power values of our test for testing dependence of ARCH(1) model.


\subsubsection{Vandermonde matrix}
Consider the $n\times p$ vandermonde matrix $\bbV$ of the form
\begin{equation}
\bbV=\frac{1}{\sqrt{n}}\left(
\begin{array}{cccc}
1 & 1 & \cdots & 1 \\
e^{-i\omega_1} & e^{-i\omega_2}& \cdots & e^{-i\omega_p} \\
\vdots & \vdots & \ddots & \vdots \\
e^{-i(n-1)\omega_1} & e^{-i(n-1)\omega_2} & \cdots & e^{-i(n-1)\omega_{p}}
\end{array}
\right)
\end{equation}
where $\omega_i$ for $i=1,\ldots,p$ are called phased distributions and are assumed to be i.i.d on $[0,2\pi)$. Then, the entries of $\bbV$ lie on the unit circle. Obviously, all the entries of the rows of $\bbV$ are not independent while the columns are independent. Denote the sample covariance matrix of $\bbV$ by $\bbD=\bbV^{H}\bbV$.

Vandermonde matrices play an important role in signal processing and wireless applications, such as direction of arrival estimation, pre--coding or sparse sampling theory. \cite{RD2009} have established that as both $n, p$ go to $\infty$ with their ratio being a positive constant, the limiting spectral distribution of $\bbD=\bbV^{H}\bbV$ is not the M-P law. This result reminds us that the proposed test should be applied to capture the dependence structure of the rows of the matrix $\bbV$. It is easy to see that, for any $k=1,\ldots,n-1$ and $j=1,\ldots,p$, $E(e^{-ik\omega_{j}})^2=0$ and $E|e^{-ik\omega_{j}}|^4=1$. The empirical power values given in Table \ref{tb7} show that the proposed test works well in detecting dependence of Vandermonde matrices.

\section{Empirical analysis of financial data}

As an application of the proposed independence test, we test whether there is any cross-sectional independence among the daily closed stock prices of some relevant companies from New York Stock Exchange (NYSE) during the period $1990.1.1-2002.1.1$. The data set is obtained from Wharton Research Data Services (WRDS) database. It consists of daily closed stock prices of $632$ companies at $5001$ days.

We randomly choose $p$ companies from the total $632$ companies in the data set. For each company $i=1,2,\ldots,p$, there exist $5001$ daily closed stock prices and denote it as $\bbs_i=(s_{i1},s_{i2},\ldots,s_{i5001})$. From the price vector $\bbs_i$, we construct a time series of closed stock prices with length $n$: $\bbx_i=(s_{i1}, s_{i,1+50}, s_{i,1+2\times 50}, \ldots, s_{i,1+50n})$. The construction of $\bbx_i, i=1,\ldots,p$ is based on the idea that the prices tend to be independent as the length of the time between them is large. Hence for each company $i$, the elements of stock price time series $\bbx_i$ are independent.

The proposed test $M_n$ is applied to test cross--sectional independence of $\bbx_1,\bbx_2,\ldots,\bbx_p$. For each $(p,n)$, we randomly choose $p$ companies from the database, construct the corresponding vectors $\bbx_1,\bbx_2,\ldots,\bbx_p$, and then calculate the P-value of the proposed test. Repeat this procedure $100$ times and plot the P-value graph to see whether the cross-sectional 'dependence' feature is popular in NYSE.

From Figure \ref{fig1}, we can see that, as the number of companies $p$ increases, more experiments are rejected in terms of the P-values below $0.05$. When $p=30, n=35$, all the $100$ experiments are rejected. This phenomenon is reasonable since as $p$ increases, the opportunity that we choose dependent companies increases as the number of total companies is invariant. It shows that cross-sectional dependence exists and is popular in NYSE. This suggests that the assumption that cross-sectional independence in such empirical studies may not be appropriate.

As comparison, we select $40$ companies from the transportation section of NYSE and investigate the dependence of daily stock prices of these companies from the same industry section--transportation section. We also perform the same experiments as above for those total $40$ companies. The P-value graphs of $n=10$, $p=5$, $n=15$, $p=10$ and $n=20$, $p=15$ are provided in Figure \ref{fig2}. Compared with the corresponding P-value graphs $n=10$, $p=5$, $n=15$, $p=10$ and $n=20$, $p=15$ in Figure \ref{fig1}, the daily closed stock prices of companies from the same section are more likely to be dependent.

\section{Conclusions}

This paper provides an approach for testing independence among a large number of high dimensional random vectors based on the characteristic function of the empirical spectral distribution of the sample covariance matrix of the random vectors. This test can capture various kinds of dependent structures, such as. MA(1), AR(1) model, nonlinear MA(1) model, ARCH(1) model and the general panel data model established in the simulation section. The conventional method (LRT proposed by \cite{Ader1984}) utilises the correlated relationship between the random vectors to capture their dependence. This idea is only efficient for normally distributed data. It may be an inappropriate tool for non-Gaussian distributed data, such as martingale difference sequences (such as, ARCH(1) model), nonlinear MA(1) model, and the Vandermonde matrix, which possess dependent but uncorrelated structures. The proposed test is not restricted to normally distributed data. Moreover, it is designed for a large number of high dimensional random vectors. 
An empirical application to closed stock prices of several companies from New York Stock Exchange highlights this approach.

\section{Appendix}
\noindent{\bf A.1 \ Some useful lemmas}
\medskip

\begin{lem}[Theorem $8.1$ of \cite{Bill1999}]\label{lem2}
Let $P_n$ and $P$ be probability measures on $(C,\mathcal{\varphi})$. If the finite dimensional distributions of $P_n$ converge weakly to those of $P$, and if $\{P_n\}$ is tight, then $P_n\Rightarrow P$.
\end{lem}

\begin{lem}[Theorem $12.3$ of \cite{Bill1999}]\label{lem3}
The sequence $\{X_n\}$ is tight if it satisfies these two conditions

(I) The sequence $\{X_n(0)\}$ is tight.

(II) There exists constants $\gamma\geq 0$, $\alpha>1$, and a nondecreasing, continuous function $F$ on $[0,1]$ such that
\begin{equation}
E\{|X_n(t_2)-X_n(t_1)|^{\gamma}\}\leq |F(t_2)-F(t_1)|^{\alpha}
\end{equation}
holds for all $t_1, t_2$, and $n$.
\end{lem}

\begin{lem}[Continuous Theorem]\label{lem4}
Let $X_n$ and $X$ be random elements defined on a metric space $S$. Suppose $g:S\rightarrow S^{'}$ has a set of discontinuous points $D_g$ such that $P(X\in D_g)=0$. Then
\begin{equation}
X_n\overset{d}{\rightarrow} X \Rightarrow g(X_n)\overset{d}{\rightarrow} g(X).
\end{equation}
\end{lem}

\begin{lem}[Complex mean value theorem (see Lemma $2.4$ of \cite{GN2006})]\label{lem5}
Let $\Omega$ be an open convex set in $\mathbb{C}$. If $f: \Omega\rightarrow \mathbb{C}$ is an analytic function and $a,b$ are distinct points in $\Omega$, then there exist points $u,v$ on $L(a,b)$ such that
\begin{equation}
Re(\frac{f(a)-f(b)}{a-b})=Re(f^{'}(u)),\
Im(\frac{f(a)-f(b)}{a-b})=Im(f^{'}(v)),
\end{equation}
where $Re(z)$ and $Im(z)$ are the real and imaginary parts of $z$ respectively; and $L(a,b)\triangleq \{a+t(b-a): t\in (0,1)\}$.
\end{lem}

\begin{lem}[Theorem A.44 of \cite{BS2009}]\label{lem6}
Let $\bbA$ and $\bbB$ be two $p\times n$ complex matrices. Then,
\begin{eqnarray}
||F^{\bbA\bbA^{*}}-F^{\bbB\bbB^{*}}||\leq \frac{1}{p}rank(\bbA-\bbB).
\end{eqnarray}
\end{lem}

\noindent{\bf A.2 \ Proofs of main theorems}
\medskip

\cite{BS2004} established the remarkable central limit theorem for functional of eigenvalues of $\textbf{A}_n$ under the additional assumption that $E|X_{11}|^4=3$ while \cite{PZ2008} provided a supplement to this theorem by eliminating the condition to some extent. From Theorem $4$ of \cite{PZ2008} and $(2.11)$ we can directly obtain the following lemma.

\begin{lem}\label{thm1}
Under Assumptions 1 and 2, we have,
for any positive integer $k$,
\begin{equation}\label{se1}
\Big(\int cos(t_1x)dG_n(x), \ldots, \int cos(t_kx)dG_n(x), \int sin(t_1x)dG_n(x), \ldots, \int sin(t_kx)dG_n(x)\Big)
\end{equation}
converges in distribution to the Gaussian vector of the form $$(V(t_1)+\delta_1(t_1), \ldots, V(t_k)+\delta_1(t_k), Z_1+\delta_2(t_1), \ldots, Z(t_k)+\delta_2(t_k)),$$ where $\delta_1(t),\ \delta_2(t)$ are, respectively, defined as
\begin{equation}\label{a7}
\delta_1(t)=\lim\limits_{n\rightarrow\infty}\int cos(tx)dp(F^{c_n,H_n}(x)-F^{c_n}(x)),
\end{equation}
\begin{equation}\label{a8}
\delta_2(t)=\lim\limits_{n\rightarrow\infty}\int sin(tx)dp(F^{c_n,H_n}(x)-F^{c_n}(x)).
\end{equation}

The means and covariances of $V(t_j)$ and $Z(t_j)$ are specified in Theorem 1.

\end{lem}

\begin{rmk}
When $\bbT_n=\bbI$, the mean and variance of the asymptotic Gaussian distribution for power functions $f(x)=x^{r}$, $\forall r\in \mathbb{Z}^{+}$ is calculated in \cite{PZ2008} and \cite{BS2004}. Hence the corresponding means and covariances for $f_1(x)=sintx$ and $f_2(x)=costx$ can be derived by Taylor series of $sintx$ and $costx$.
\end{rmk}

\begin{proof}[Proof of Theorem 1]

Let $t$ belong to a closed interval $I=[T_1,T_2]$. To finish Theorem 1, in view of Lemma \ref{lem2} and Lemma \ref{thm1}, it suffices to prove the tightness of $\{\big(\phi_n(t), \psi_n(t)\big): t\in I\}$. Thus it suffices to prove the tightness of $p(s_n(t)-s(t))$. Repeating the same truncation and centralization steps as those in \cite{BS2004}, we may assume that
\begin{equation}
|X_{ij}|<\delta_n\sqrt{n},\  EX_{ij}=0,\  E|X_{ij}|^2=1,\  E|X_{ij}|^{4}<\infty.
\end{equation}

Set $M_n(z)=n[m_{F^{\textbf{A}_n}}(z)-m_{F^{c_n,H_n}}(z)]$. By the Cauchy theorem
\begin{equation}
f(x)=-\frac{1}{2\pi i}\oint \frac{f(z)}{z-x}dz,
\end{equation}
we have, with probability one, for all $n$ large,
\begin{equation}
\int e^{itx}dp(F^{\textbf{A}_n}(x)-F^{c_n,H_n}(x))=-\frac{1}{2\pi i}\oint_{\mathcal{C}} e^{itz}M_n(z)dz.
\end{equation}

The contour $\mathcal{C}$ involved in the above integral is specified as follows. Let
\begin{equation}
\mathcal{C}_u=\{x+iv_0: x\in[x_l,x_r]\},
\end{equation}
where $v_0>0$, $x_r$ is any number greater than $\limsup\limits_{n}\lambda_{max}(\bbT_n)(1+\sqrt{c})^{2}$, $x_l$ is any negative number if $c\geq 1$ and otherwise choose $x_l\in(0,\limsup\limits_{n}\lambda_{min}(\bbT_n)(1-\sqrt{c})^{2})$. Then the contour $\mathcal{C}$ is defined by the union of $\mathcal{C}_+$ and its symmetric part $\mathcal{C}_{-}$ with respect to the $x$-axis, where
\begin{equation}
\mathcal{C}_+=\{x_l+iv: v\in[0,v_0]\}\cup\mathcal{C}_u\cup\{x_r+iv: v\in[0,v_0]\}.
\end{equation}

From Lemma 1 and the argument regarding equivalence in probability of $M_n(z)$ and its truncation version given in Page 563 in \cite{BS2004} and Lemma 3 we have
\begin{equation}\label{r1}
\oint_{\mathcal{C}}|M_n(z)||dz|\xrightarrow{D} \oint_{\mathcal{C}}|M(z)||dz|,
\end{equation}
where $M(z)$ is a Gaussian process, the limit of $M_n(z)$. 

We conclude from Lemma \ref{lem5} that for any $\delta>0$
\begin{eqnarray}\label{a11}
&&\sup_{|t_1-t_2|<\delta, t_1,t_2\in I}\big|\oint_{\mathcal{C}}(e^{it_1z}-e^{it_2z})M_n(z)dz\big|\non
&\leq& \sup_{|t_1-t_2|<\delta, t_1,t_2\in I}\Big|\oint_{\mathcal{C}}\sqrt{\big(Re(ize^{it_3z})\big)^2+\big(Im(ize^{it_4z})\big)^2}\delta | M_n(z)||dz|\Big|\non
&\leq& K\delta\Big|\oint_{\mathcal{C}}|M_n(z)||dz|\Big|
\stackrel{D}\longrightarrow K\delta\Big|\oint_{\mathcal{C}}|M(z)||dz|\Big|,\quad \text{as} \quad n\rightarrow\infty,
\end{eqnarray}
where $t_3$ and $t_4$ lies in the interval $[T_1,T_2]$, the last inequality uses (\ref{r1}) and the fact that $Re(ize^{it_3z})$, $Im(ize^{it_4z})$ are bounded on the contour $\mathcal{C}$ and $K$ (and in the sequel) is a constant number which may be different from line to line.

By (\ref{a11}), we have for any $\varepsilon>0$,
\begin{eqnarray}\label{a12}
P\Big(\sup_{|t_1-t_2|<\delta, t_1,t_2\in [0,1]}\Big|\oint_{\mathcal{C}}(e^{it_1z}-e^{it_2z})M_n(z)dz\Big|\geq\varepsilon\Big)
\leq P\Big(K\delta\Big|\oint_{\mathcal{C}}|M_n(z)||dz|\Big|\geq\varepsilon\Big)
\end{eqnarray}
and
\begin{eqnarray}\label{a13}
\lim_{\delta\rightarrow 0}\limsup_{n\rightarrow \infty}P\Big(K\delta\Big|\oint_{\mathcal{C}}|M_n(z)||dz|\Big|\geq\varepsilon\Big)
=\lim_{\delta\rightarrow 0}P\Big(K\delta\Big|\oint_{\mathcal{C}}|M(z)||dz|\Big|\geq\varepsilon\Big)=0.
\end{eqnarray}
Hence (\ref{a12}) and (\ref{a13}) imply that
\begin{eqnarray}\label{a20}
\lim_{\delta\rightarrow 0}\limsup_{n\rightarrow \infty}P\Big(\sup_{|t_1-t_2|<\delta, t_1,t_2\in I}\Big|\oint_{\mathcal{C}}(e^{it_1z}-e^{it_2z}) M_n(z)dz\Big|\geq\varepsilon\Big)=0.
\end{eqnarray}

By Theorem $7.3$ of \cite{Bill1999}, $\int e^{itx}dp(F^{\textbf{A}_n}(x)-F^{c_n,H_n}(x))$ is tight. Moreover from the assumption we see that $\int e^{itx}dp(F^{c_n,H_n}(x)-F^{c_n}(x))$ is tight by Lemma 4.

\end{proof}

To prove Theorem 2, we first need to establish Lemma \ref{thm3} below. To this end, write
\begin{equation}
H_n(x)=p[F^{\bbS_n}(x)-F^{c_n}(x)],
\end{equation}
where $\bbS_n$ is defined in $(4.4)$.

\begin{lem}\label{thm3}
Under the assumptions of Theorem 2, we have
for any positive integer $k$,
\begin{eqnarray}
\Big(\int cos(t_1x)dH_n(x), \ldots, \int cos(t_kx)dH_n(x), \int sin(t_1x)dH_n(x), \ldots, \int sin(t_kx)dH_n(x)\Big)
\end{eqnarray}
converges in distribution to a Gaussian vector $(W(t_1),\ldots,W(t_k),Q(t_1),\ldots,Q(t_k))$ whose mean and covariance function are given in Theorem 2.
\end{lem}

\begin{proof}[Proof of Lemma \ref{thm3}]

Using the same truncation and centralization steps as those used in the paper by \cite{BS2004}, we may assume that
\begin{equation}
|\varepsilon_{ij}|<\delta_n\sqrt{n},\  E\varepsilon_{ij}=0,\  E|\varepsilon_{ij}|^2=1,\  E|\varepsilon_{ij}|^{4}<\infty.
\end{equation}

For the panel data model (4.1) proposed in Section 4, let
\begin{eqnarray}
\bbv_j=(v_{1j},\ldots,v_{pj})^{T}, \ \boldsymbol\varepsilon_j=(\boldsymbol\varepsilon_{1j},\ldots,\boldsymbol\varepsilon_{pj})^{T}, \ \bbu=(\frac{1}{\sqrt{p}}u_1,\ldots,\frac{1}{\sqrt{p}}u_p)^{T},\ j=1,\ldots,n.
\end{eqnarray}
The model can be written in the vector form as
\begin{equation}\label{r4}
\bbv_j=\boldsymbol\varepsilon_j+\bbu,\ j=1,\ldots,n.
\end{equation}

We then define the sample covariance matrix by $\bbS_n=\frac{1}{n}\sum\limits_{j=1}^n\bbv_j\bbv_j^T$. Moreover write
\begin{eqnarray}
\bar\bbv=\frac{1}{n}\sum^{n}_{j=1}\bbv_j,\ \bar{\boldsymbol\varepsilon}=\frac{1}{n}\sum^{n}_{j=1}\boldsymbol\varepsilon_j,
\end{eqnarray}
and
\begin{equation}
\bbD_n=\frac{1}{n}\sum^{n}_{j=1}\boldsymbol\varepsilon_j\boldsymbol\varepsilon_j^{T}, \ \
\mathcal{S}_n=\frac{1}{n}\sum^{n}_{j=1}(\bbv_j-\bar\bbv)(\bbv_j-\bar\bbv)^{T},\ \
\mathcal{D}_n=\frac{1}{n}\sum^{n}_{j=1}(\boldsymbol\varepsilon_j-\bar{\boldsymbol\varepsilon})(\boldsymbol\varepsilon_j-\bar{\boldsymbol\varepsilon})^{T}.
\end{equation}

Note that $\mathcal{S}_n=\mathcal{D}_n$. The sample covariance matrix $\bbS_n$ can be then expressed as
\begin{eqnarray}\label{panel14}
\bbS_n=\mathcal{S}_n+\bar\bbv\bar\bbv^{T}=\mathcal{D}_n+\bar\bbv\bar\bbv^{T}.
\end{eqnarray}

By the conditions of Theorem 2 and the Burkholder inequality we have
$$
E|\bbu^{T}\bar{\boldsymbol\varepsilon}|^4\leq \frac{1}{n^4}E|\sum^{n}_{j=1}\bbu^{T}\boldsymbol\varepsilon_j|^4\leq \frac{K}{n^4}E|\sum^{n}_{j=1}\bbu^{T}\bbu|^2+
\frac{K}{n^4}\sum^{n}_{j=1}E|\bbu^{T}\boldsymbol\varepsilon_j|^4=O(\frac{1}{n^2}),
$$
which, together with Borel-Cantelli's Lemma, implies that
$$
\bbu^{T}\bar{\boldsymbol\varepsilon}\stackrel{a.s.}\longrightarrow 0.
$$

Also, the conditions of Theorem 2 imply that
\begin{equation}
\bbu^{T}\bbu\rightarrow \bar{u}.\label{h1}
\end{equation}

Therefore, by the conditions of Theorem 2 and Theorem $2$ of \cite{PZ2011}, we have, as $n\rightarrow\infty$,
\begin{eqnarray}\label{panel15}
\lambda_{max}(\bar\bbv\bar\bbv^{T})=\bar\bbv^{T}\bar\bbv=\bar{\boldsymbol\varepsilon}^{T}\bar{\boldsymbol\varepsilon}+\bbu^{T}\bbu+2\bbu^{T}\bar{\boldsymbol\varepsilon}
\stackrel{a.s.}\rightarrow c+\bar u, \ \ as \ n\rightarrow\infty.
\end{eqnarray}

Furthermore, \cite{Jiang2004} proves that
\begin{equation}\label{panel17}
\lambda_{max}(\mathcal{D}_n)\rightarrow (1+\sqrt{c})^2, \ a.s.\ as\ n\rightarrow\infty
\end{equation}
and \cite{XZ2010} implies that, when $c\leq 1$
\begin{eqnarray}\label{panel18}
\lambda_{min}(\mathcal{D}_n)\rightarrow (1-\sqrt{c})^2, \ a.s.\ as\ n\rightarrow\infty.
\end{eqnarray}

By (\ref{panel15}) (\ref{panel17}) and (\ref{panel18}),  the maximal and minimal eigenvalues of $\bbS_n$ satisfy with probability one
\begin{eqnarray}
\limsup_{n\rightarrow\infty}\lambda_{max}(\bbS_n)\leq c+\bar u+(1+\sqrt{c})^2,
\end{eqnarray}
and
\begin{eqnarray}
\liminf_{n\rightarrow\infty}\lambda_{min}(\bbS_n)\geq(1-\sqrt{c})^2.
\end{eqnarray}

As in the proof of Theorem $2$, we obtain from Cauchy's formula, with probability one, for $n$ large,
\begin{eqnarray}\label{panel11}
p\int f(x)d[F^{\bbS_n}(x)-F^{c_n}(x)]&=&\frac{p}{2\pi i}\int\oint_{\gamma}\frac{f(z)}{z-x}dz d[F^{\bbS_n}(x)-F^{c_n}(x)]\non
&=&\frac{p}{2\pi i}\oint_{\gamma}f(z)dz\int\frac{1}{z-x}d[F^{\bbS_n}(x)-F^{c_n}(x)]\non
&=&-\frac{1}{2\pi i}\oint_{\gamma}f(z)(tr(\bbS_n-z\bbI_p)^{-1}-pm_{c_n}(z))dz,
\end{eqnarray}
where $m_{c_n}(z)$ is obtained from $m(z)$ with $c$ replaced by $c_n$. The contour $\gamma$ is specified as follows: Let $v_0>0$ be arbitrary and set $\gamma_{\mu}=\{\mu+iv_0, \mu\in[\mu_{\ell},\mu_r]\}$,
where $\mu_r>c+\bar u+(1+\sqrt{c})^2$ and $0<\mu_{\ell}<I_{(0,1)}(c)(1-\sqrt{c})^2$ or $\mu_{\ell}$ is any negative number if $c\geq 1$. Then define
\begin{eqnarray}
\gamma^{+}=\{\mu_{\ell}+iv: v\in[0,v_0]\}\cup\gamma_{\mu}\cup\{\mu_r+iv: v\in[0,v_0]\}
\end{eqnarray}
and let $\gamma^{-}$ be the symmetric part of $\gamma^{+}$ about the real axis. Then set $\gamma=\gamma^{+}\cup\gamma^{-}$.

Set
\begin{eqnarray}
&&\mathcal{S}_n^{-1}(z)=(\mathcal{S}_n-z\bbI_p)^{-1}, \ \bbS_n^{-1}(z)=(\bbS_n-z\bbI_p)^{-1},\non
&&\mathcal{D}_n^{-1}(z)=(\mathcal{D}_n-z\bbI_p)^{-1}, \ \bbD_n^{-1}(z)=(\bbD_n-z\bbI_p)^{-1}.
\end{eqnarray}
Then we have
\begin{eqnarray}\label{panel2}
tr\bbS_n^{-1}(z)-pm_{c_n}(z)=\big(tr\mathcal{D}_n^{-1}(z)-pm_{c_n}(z)\big)-\frac{\bar\bbv^{T}\mathcal{D}_n^{-2}(z)\bar\bbv}{1+\bar\bbv^{T}\mathcal{D}_n^{-1}(z)\bar\bbv},
\end{eqnarray}
where we have used the identity
\begin{eqnarray}\label{b5}
(\bbC+\bbr\bbr^{T})^{-1}=\bbC^{-1}-\frac{\bbC^{-1}\bbr\bbr^{T}\bbC^{-1}}{1+\bbr^{T}\bbC^{-1}\bbr},
\end{eqnarray}
where $\bbC$ and $(\bbC+\bbr\bbr^{T})$ are both invertible; and $\bbr\in\mathbb{R}^{p}$. The first term on the right hand of (\ref{panel2}) was investigated in \cite{P2011}. In what follows we consider the second term on the right hand of (\ref{panel2}).

One may verify that
\begin{equation}
(\bbC+q\bbr\bbv^T)^{-1}=\frac{\bbC^{-1}}{1+q\bbv^T\bbC^{-1}\bbr},\label{r3}
\end{equation}
where  $\bbC$ and $(\bbC+q\bbr\bbv^T)$ are both invertible, $q$ is a scalar and $\bbr,\bbv\in\mathbb{R}^{p}$. This, together with (\ref{r4}) and (\ref{b5}), yields
\begin{eqnarray}\label{panel3}
\bar\bbv^{T}\mathcal{D}_n^{-1}(z)\bar\bbv&=&\bar{\boldsymbol\varepsilon}^{T}\mathcal{D}_n^{-1}(z)\bar{\boldsymbol\varepsilon}+2\bbu^{T}\mathcal{D}_n^{-1}(z)\bar{\boldsymbol\varepsilon}
+\bbu^{T}\mathcal{D}_n^{-1}(z)\bbu\non
&=&\frac{\bar{\boldsymbol\varepsilon}^{T}\bbD_n^{-1}(z)\bar{\boldsymbol\varepsilon}}{1-\bar{\boldsymbol\varepsilon}^{T}\bbD_n^{-1}(z)\bar{\boldsymbol\varepsilon}}
+2\frac{\bbu^{T}\bbD_n^{-1}(z)\bar{\boldsymbol\varepsilon}}
{1-\bar{\boldsymbol\varepsilon}\bbD_n^{-1}(z)\bar{\boldsymbol\varepsilon}}+\bbu^{T}\mathcal{D}_n^{-1}(z)\bbu
\end{eqnarray}
and
\begin{eqnarray}\label{panel4a}
&&\bar\bbv^{T}\mathcal{D}_n^{-2}(z)\bar\bbv=\bar{\boldsymbol\varepsilon}^{T}\mathcal{D}_n^{-2}(z)\bar{\boldsymbol\varepsilon}+2\bbu^{T}\mathcal{D}_n^{-2}(z)\bar{\boldsymbol\varepsilon}
+\bbu^{T}\mathcal{D}^{-2}_n(z)\bbu
\label{panel4}\\
&=&\frac{\bar{\boldsymbol\varepsilon}^{T}\bbD_n^{-2}(z)\bar{\boldsymbol\varepsilon}}{(1-\bar{\boldsymbol\varepsilon}
\bbD_n^{-1}(z)\bar{\boldsymbol\varepsilon})^2}+\frac{2\bbu^{T}\bbD_n^{-2}(z)\bar{\boldsymbol\varepsilon}}{1-\bar{\boldsymbol\varepsilon}^{T}
\bbD_n^{-1}(z)\bar{\boldsymbol\varepsilon}}
+\frac{2\bbu^{T}\bbD_n^{-1}(z)\bar{\boldsymbol\varepsilon}\bar{\boldsymbol\varepsilon}^{T}\bbD_n^{-2}(z)\bar{\boldsymbol\varepsilon}}
{(1-\bar{\boldsymbol\varepsilon}^{T}\bbD_n^{-1}(z)\bar{\boldsymbol\varepsilon})^2}+\bbu^{T}\mathcal{D}^{-2}(z)\bbu.
\nonumber
\end{eqnarray}

It is proved in Section 2.5 and (4.3) of \cite{P2011} that as $n\rightarrow\infty$,
\begin{eqnarray}\label{panel5}
\sup_{z\in\gamma}\Big|\bar{\boldsymbol\varepsilon}^{T}\bbD_n^{-2}(z)\bar{\boldsymbol\varepsilon}-\frac{c\underline{m}^2(z)}{(1+\underline{m}(z))^2-c\underline{m}^2(z)}\Big|
\xrightarrow{i.p.} 0;
\end{eqnarray}
\begin{eqnarray}\label{panel6}
\sup_{z\in\gamma}\Big|\bar{\boldsymbol\varepsilon}^{T}\bbD_n^{-1}(z)\bar{\boldsymbol\varepsilon}-(1+z\underline{m}(z))\Big|\xrightarrow{i.p.} 0;
\end{eqnarray}
and
\begin{eqnarray}\label{panel7}
\sup_{z\in\gamma}\Big|\bbu^{T}\bbD_n^{-1}(z)\bar{\boldsymbol\varepsilon}\Big|\xrightarrow{i.p.} 0,
\end{eqnarray}
where we have also used an argument similar to (2.28) of \cite{P2011}).

By Lemma $2$ of \cite{BMP2007}, (3.4) and (4.3) in \cite{P2011}), and $(4.5)$ we have, as $n\rightarrow\infty$,
\begin{eqnarray}\label{panel8}
\sup_{z\in\gamma}\Big|\bbu^{T}\mathcal{D}_n^{-1}(z)\bbu-\bar u\underline{m}(z)\Big|\xrightarrow{i.p.} 0.
\end{eqnarray}

The next aim is to prove that
\begin{eqnarray}\label{panel9}
\sup_{z\in\gamma}\Big|\bbu^{T}\mathcal{D}_n^{-2}(z)\bbu-\bar u\int\frac{1}{(\lambda-z)^2}dF^{MP}(\lambda)\Big|\xrightarrow{i.p.} 0
\end{eqnarray}
and that
\begin{eqnarray}\label{panel7*}
\sup_{z\in\gamma}\Big|\bbu^{T}\bbD_n^{-2}(z)\bar{\boldsymbol\varepsilon}\Big|\xrightarrow{i.p.} 0.
\end{eqnarray}

Consider (\ref{panel9}) first. By the formula (\ref{b5}), we have an expansion
\begin{eqnarray*}
\bbu^{T}\mathcal{D}_n^{-2}(z)\bbu=\bbu^{T}\bbD^{-2}_n(z)\bbu+\frac{2\bbu^{T}\bbD^{-2}_n(z)\bar{\boldsymbol{\varepsilon}}\bar{\boldsymbol{\varepsilon}}^{T}
\bbD^{-1}_n(z)\bbu}{1-\bar{\boldsymbol{\varepsilon}}^{T}\bbD^{-1}_n(z)\bar{\boldsymbol{\varepsilon}}}+\frac{\bbu^{T}\bbD_n^{-1}(z)\bar{\boldsymbol{\varepsilon}}\bar{\boldsymbol{\varepsilon}}^{T}
\bbD^{-2}_n(z)\bar{\boldsymbol{\varepsilon}}\bar{\boldsymbol{\varepsilon}}^{T}\bbD_n^{-1}(z)\bbu}{(1-\bar{\boldsymbol{\varepsilon}}^{T}\bbD^{-1}_n(z)\bar{\boldsymbol
{\varepsilon}})^2}.
\end{eqnarray*}

For any given $z\in\gamma$, we conclude from Theorem $1$ of \cite{P2011} and Helly-Bray's theorem that
\begin{eqnarray}
\bbu^{T}\mathcal{D}^{-2}_n(z)\bbu-\bar u\int\frac{1}{(\lambda-z)^2}dF^{MP}(\lambda)\stackrel{i.p.}\longrightarrow 0\quad \text{as}\quad n\rightarrow\infty.
\end{eqnarray}

By the expansion of $\bbu^{T}\mathcal{D}^{-2}_n(z)\bbu$ and (\ref{panel5})-(\ref{panel7}), to prove (\ref{panel9}), it suffices to prove the tightness of $\Big\{K_n^{(1)}(z)=\bbu^{T}\bbD^{-2}_n(z)\bbu-\bar u\int\frac{1}{(\lambda-z)^2}dF^{MP}(\lambda), z\in\gamma\Big\}$ and $\Big\{\bbu^{T}\bbD^{-2}_n(z)\bar{\boldsymbol{\varepsilon}}, z\in\gamma\Big\}$.

To this end,  as in \cite{BS2004}, below introduce the truncated version of $\bbu^{T}\bbD_n^{-2}(z)\bbu$. Define $\gamma_{r}=\{\mu_r+iv: v\in [n^{-1}\rho_n, v_0]\}$,
\begin{equation}
\gamma_{\ell}=\left\{\begin{array}{cc}
                     \{\mu_{\ell}+iv: v\in [n^{-1}\rho_n, v_0]\}, & \mu_{\ell}>0, \\
                     \{\mu_{\ell}+iv: v\in [0, v_0]\}, & \mu_{\ell}<0,
                   \end{array}
\right.
\end{equation}
where
\begin{eqnarray}
\rho_{n}\downarrow 0, \ \ \rho_{n}\geq n^{-\theta}, \ for\ some\ \theta\in (0,1).
\end{eqnarray}

Let $\gamma_n^{+}=\gamma_{\ell}\cup\gamma_{\mu}\cup\gamma_r$ and $\gamma_n^{-}$ denote the symmetric part of $\gamma_n^{+}$ with respect to the real axis. We then define the truncated process $\widehat{\bbu^{T}\bbD_n^{-2}(z)\bbu}$ of the process $\bbu^{T}\bbD_n^{-2}(z)\bbu$ for $z=\alpha+iv$ by
\begin{eqnarray}\label{r7}
\widehat{\bbu^{T}\bbD_n^{-2}(z)\bbu}
=\left\{\begin{array}{cc}
        \bbu^{T}\bbD_n^{-2}(z)\bbu & z\in\gamma_n=\gamma^{+}_n\cup\gamma^{-}_n, \\
         \frac{nv+\rho_n}{2\rho_n}\bbu^{T}\bbD_n^{-2}(z_{r_1})\bbu+ \frac{\rho_n-nv}{2\rho_n}\bbu^{T}\bbD_n^{-2}(z_{r_2})\bbu & \mu=\mu_r, v\in [-n^{-1}\rho_n, n^{-1}\rho_n],\\
         \frac{nv+\rho_n}{2\rho_n}\bbu^{T}\bbD_n^{-2}(z_{\ell_1})\bbu+ \frac{\rho_n-nv}{2\rho_n}\bbu^{T}\bbD_n^{-2}(z_{\ell_2})\bbu & \mu=\mu_{\ell}>0, v\in [-n^{-1}\rho_n, n^{-1}\rho_n],
       \end{array}
\right.
\end{eqnarray}
where $z_{r_1}=\mu_r+in^{-1}\rho_n$, $z_{r_2}=\mu_r-in^{-1}\rho_n$, $z_{\ell_1}=\mu_{\ell}+in^{-1}\rho_n$ and $z_{\ell_2}=\mu_{\ell}-in^{-1}\rho_n$. We then have
\begin{equation}\label{r8}
\sup_{z\in\gamma}\Big|\widehat{\bbu^{T}\bbD_n^{-2}(z)\bbu}-\bbu^{T}\bbD_n^{-2}(z)\bbu\Big| \leq K\rho_n\|\bbu\|^2(\frac{1}{|\lambda_{\max}(\bbD_n)-\mu_r|}+\frac{1}{|\lambda_{\min}(\bbD_n)-\mu_l|})\stackrel{i.p.}\longrightarrow 0.
\end{equation}

It is proved in Section $3$ of \cite{BS2004} that, for any positive integer $k$ and $z\in\gamma_n^{+}\cup\gamma_n^{-}$,
\begin{eqnarray}
max\big(E||D^{-1}_n(z)||^k\big)\leq K.
\end{eqnarray}
It follows from independence between $\bbu$ and $\boldsymbol\varepsilon_j,j=1,\cdots,n$ that
\begin{eqnarray}\label{a15}
&&E\big|\bbu^{T}\bbD_n^{-2}(z)\bbu-\bar u\int\frac{1}{(\lambda-z)^2}dF^{MP}(\lambda)\big|\non
&\leq& E|\bbu^{T}\bbD^{-2}_n(z)\bbu|+\Big|\bar u \int\frac{1}{(\lambda-z)^2}dF^{MP}(\lambda)\Big|\non
&\leq& E||\bbu^{T}||^2E||\bbD_n^{-2}(z)||^2+K\leq K,
\end{eqnarray}
which ensures Condition $(1)$ of Lemma \ref{lem3}. Similarly, we can derive $E|\bbu^{T}\bbD_n^{-2}(z)\bar{\boldsymbol{\varepsilon}}|^2\leq K$.

Next, we prove condition $(2)$ of Lemma \ref{lem3}, i.e.
\begin{eqnarray}\label{a14}
\sup_{n,z_1,z_2\in\gamma_n^{+}\cup\gamma_n^{-}}\frac{E|K^{(i)}_n(z_1)-K^{(i)}_n(z_2)|^2}{|z_1-z_2|^2}<\infty, \quad i=1,2.
\end{eqnarray}

Note that
\begin{eqnarray}
\bbA^{-1}-\bbB^{-1}=\bbA^{-1}(\bbB-\bbA)\bbB^{-1},
\end{eqnarray}
where $\bbA$ and $\bbB$ are any two nonsingular matrices. We then conclude that
\begin{equation}\label{r5}
\bbD^{-2}_n(z_1)-\bbD^{-2}_n(z_2)=(z_1-z_2)\bbD^{-2}_n(z_1)\bbD_n^{-1}(z_2)+(z_1-z_2)\bbD^{-1}_n(z_1)\bbD_n^{-2}(z_2).
\end{equation}

Then
\begin{eqnarray}
\frac{K^{(1)}_n(z_1)-K^{(1)}_n(z_2)}{z_1-z_2}&=&\bbu^{T}\bbD^{-2}_n(z_1)\bbD^{-1}_n(z_2)\bbu+\bbu^{T}\bbD^{-1}_n(z_1)\bbD^{-2}_n(z_2)\bbu\non
&&-\bar u\int\frac{(\lambda-z_1)+(\lambda-z_2)}{(\lambda-z_1)^2(\lambda-z_2)^2}dF^{MP}(\lambda).
\end{eqnarray}

As in (\ref{a15}), we can obtain
\begin{eqnarray}\label{a16}
E|\bbu^{T}\bbD^{-2}_n(z_1)\bbD^{-1}_n(z_2)\bbu|^2\leq K, \quad E|\bbu^{T}\bbD_n^{-1}(z_1)\bbD^{-2}_n(z_2)\bbu|^2\leq K.
\end{eqnarray}

Since $f(\lambda)=\frac{(\lambda-z_1)+(\lambda-z_2)}{(\lambda-z_1)^2(\lambda-z_2)^2}$ is a continuous function when $z_1,z_2\in\gamma$, the integral
$\int f(\lambda)dF^{MP}(\lambda)$ is bounded. This, together with (\ref{a16}), implies
\begin{equation}\label{r6}
\sup_{n,z_1,z_2\in\gamma}\frac{E|K^{(1)}_n(z_1)-K^{(1)}_n(z_2)|^2}{|z_1-z_2|^2}<\infty.
\end{equation}

By (3.17) in \cite{PZ2011} and an argument similar to (\ref{r5})-(\ref{r6})  we may verify that $\bbu^{T}\bbD_n^{-2}(z)\bar{\boldsymbol\varepsilon}$ is tight for $z\in\gamma$.
Summarizing the above we obtain (\ref{panel9}).

Consider (\ref{panel7*}) now.  From the last paragraph we see that it is enough to consider the pointwise convergence of $\bbu^{T}\bbD_n^{-2}(z)\bar{\boldsymbol\varepsilon}$.  As in (\ref{r7}) and (\ref{r8}) we may define the truncated process $\widehat{\bbu^{T}\bbD_n^{-2}(z)\bar{\boldsymbol\varepsilon}}$ of the process $\bbu^{T}\bbD_n^{-2}(z)\bar{\boldsymbol\varepsilon}$ and then prove that their difference tends to zero in probability. As in (4.3) of \cite{PZ2011} one may prove that for given $z\in \gamma_n^+$ $\bbu^{T}\bbD_n^{-2}(z)\bar{\boldsymbol\varepsilon}\stackrel{i.p.}\longrightarrow 0.$

From (\ref{panel3}) to (\ref{panel7*}) we have
\begin{eqnarray}
\sup_{z\in\gamma}\Big|\bar\bbv^{T}\mathcal{D}_n^{-1}(z)\bar\bbv-\big(\frac{1+z\underline{m}(z)}{-z\underline{m}(z)}+\bar u\underline{m}(z)\big)\Big|\xrightarrow{i.p.} 0,
\end{eqnarray}
and
\begin{eqnarray}
\sup_{z\in\gamma}\Big|\bar\bbv^{T}\mathcal{D}_n^{-2}(z)\bar\bbv-\big(\frac{c}{z^2[(1+\underline{m}(z))^2-c\underline{m}^2(z)]}
+\bar u\int\frac{1}{(\lambda-z)^2}dF^{MP}(\lambda)\big)\Big|\xrightarrow{i.p.} 0.
\end{eqnarray}

We then conclude from Slutsky's theorem that
\begin{eqnarray}\label{panel10}
\sup_{z\in\gamma}\Big|\frac{\bar\bbv^{T}\mathcal{D}_n^{-2}(z)\bar\bbv}{1+\bar\bbv^{T}\mathcal{D}_n^{-1}(z)\bar\bbv}-
\frac{\frac{c\underline{m}(z)}{z[(1+\underline{m}(z))^2-c\underline{m}^2(z)]}+\bar u\int\frac{1}{(\lambda-z)^2}dF^{MP}(\lambda)}
{\bar uz\underline{m}^2(z)-1}\Big|\xrightarrow{i.p.} 0.
\end{eqnarray}


The arguments of Theorem $1$ of \cite{P2011} show that the truncation version of $\big(tr\mathcal{D}_n^{-1}(z)-pm_{c_n}(z)\big)$ converges in distribution to a two-dimensional Gaussian process and that the difference between $\big(tr\mathcal{D}_n^{-1}(z)-pm_{c_n}(z)\big)$ and its truncation version goes to zero in probability (see Page 563 of \cite{BS2004} and (2.28) of \cite{P2011}). Theorem 3 then follows from (\ref{panel10}), (\ref{panel11}) and (\ref{panel2}), Slutsky's theorem and Lemma 3 (one may refer to Page 563 of \cite{BS2004}).
\end{proof}

\begin{proof}[Proof of Theorem 2]
As in the proof of Theorem 1, in view of Lemma \ref{thm3} it suffices to prove the tightness of $\{p(s_n(t)-s(t)): t\in I\}$. As before, write
\begin{eqnarray}\label{panel11*}
p(s_n(t)-s(t))&=&p\int e^{itx}d[F^{\bbS_n}(x)-F^{c_n}(x)]\non
&=&-\frac{1}{2\pi i}\oint_{\gamma}e^{itz}(tr(\bbS_n-z\bbI_p)^{-1}-pm_{c_n}(z))dz,
\end{eqnarray}
where the contour $\gamma$ is specified in Lemma \ref{thm3}.

From the formula (\ref{b5}), we have
\begin{eqnarray}
tr\mathcal{D}_n^{-1}(z)=tr\bbD_n^{-1}(z)+\frac{\bar{\boldsymbol\varepsilon}^{T}\bbD_n^{-2}(z)\bar{\boldsymbol\varepsilon}}
{1-\bar{\boldsymbol\varepsilon}^{T}\bbD_n^{-1}(z)\bar{\boldsymbol\varepsilon}}.
\end{eqnarray}

This, together with (\ref{panel2}), yields
\begin{eqnarray}
tr\bbS_n^{-1}(z)-pm_{c_n}(z)=tr\bbD_n^{-1}(z)-pm_{c_n}(z)+\frac{\bar{\boldsymbol\varepsilon}^{T}\bbD_n^{-2}(z)\bar{\boldsymbol\varepsilon}}
{1-\bar{\boldsymbol\varepsilon}^{T}\bbD_n^{-1}(z)\bar{\boldsymbol\varepsilon}}
-\frac{\bar\bbv^{T}\mathcal{D}_n^{-2}(z)\bar\bbv}{1+\bar\bbv^{T}\mathcal{D}_n^{-1}(z)\bar\bbv}.
\end{eqnarray}

By (\ref{panel11}) and noting that $M_n(z)=tr\bbD_n^{-1}(z)-pm_{c_n}(z)$, it is sufficient to prove the tightness of the following three terms:
\begin{eqnarray}
g_{n1}(t)=-\frac{1}{2\pi i}\oint_{\gamma}e^{itz}M_n(z)dz,
\end{eqnarray}
\begin{eqnarray}
g_{n2}(t)=-\frac{1}{2\pi i}\oint_{\gamma}e^{itz}\frac{\bar{\boldsymbol\varepsilon}^{T}\bbD_n^{-2}(z)\bar{\boldsymbol\varepsilon}}
{1-\bar{\boldsymbol\varepsilon}^{T}\bbD_n^{-1}(z)\bar{\boldsymbol\varepsilon}}dz,
\end{eqnarray}
\begin{eqnarray}
g_{n3}(t)=-\frac{1}{2\pi i}\oint_{\gamma}e^{itz}\frac{\bar\bbv^{T}\mathcal{D}_n^{-2}(z)\bar\bbv}{1+\bar\bbv^{T}\mathcal{D}_n^{-1}(z)\bar\bbv}dz, 
\end{eqnarray}

The tightness of $\{g_{n1}(t): t\in I=[T_1,T_2]\}$ has been proved in Theorem 1. Next, via the same method adopted by Theorem 1, we prove the tightness of $\{g_{ni}(t): t\in I=[T_1,T_2]\}$, $i=2,3$ as follows.

By (\ref{panel5}), (\ref{panel6}) and Slutsky's theorem, we have
\begin{eqnarray}\label{a17}
\sup_{z\in\gamma}\Big|\frac{\bar{\boldsymbol\varepsilon}^{T}\bbD_n^{-2}(z)\bar{\boldsymbol\varepsilon}}
{1-\bar{\boldsymbol\varepsilon}^{T}\bbD_n^{-1}(z)\bar{\boldsymbol\varepsilon}}
+\frac{c\underline{m}(z)}{z\big((1+\underline{m}(z))^2-c\underline{m}^2(z)\big)}\Big|\xrightarrow{i.p.} 0.
\end{eqnarray}
We conclude from (\ref{a17}), (\ref{panel10}) and Lemma 3 that, as $n\rightarrow\infty$,
\begin{eqnarray}\label{a18}
\oint_{\gamma}\big|\frac{\bar{\boldsymbol\varepsilon}^{T}\bbD_n^{-2}(z)\bar{\boldsymbol\varepsilon}}
{1-\bar{\boldsymbol\varepsilon}^{T}\bbD_n^{-1}(z)\bar{\boldsymbol\varepsilon}}\big||dz|\xrightarrow{a.s.}\oint_{\gamma}
\big|\frac{c\underline{m}(z)}{z\big((1+\underline{m}(z))^2-c\underline{m}^2(z)\big)}\big||dz|
\end{eqnarray}
and
\begin{eqnarray}\label{a19}
\oint_{\gamma}\big|\frac{\bar\bbv^{T}\mathcal{D}_n^{-2}(z)\bar\bbv}{1+\bar\bbv^{T}\mathcal{D}_n^{-1}(z)\bar\bbv}\big||dz|\xrightarrow{a.s.}
\oint_{\gamma}
\Big|\frac{\frac{c\underline{m}(z)}{z[(1+\underline{m}(z))^2-c\underline{m}^2(z)]}+\bar u\int\frac{1}{(\lambda-z)^2}dF^{MP}(\lambda)}
{\bar uz\underline{m}^2(z)-1}\Big||dz|.
\end{eqnarray}

By (\ref{a18}), (\ref{a19}) and the same proof as (\ref{a11}) to (\ref{a20}), the tightness of $\{g_{ni}(t): t\in I=[T_1,T_2]\}$, $i=2,3$ can be derived.
\end{proof}

\small{

}

\begin{figure}[h]
\caption{Empirical Spectral Distribution of ARCH(1) and I.I.D Normal Vector}\label{fig0}
\centering
\includegraphics[scale=0.45]{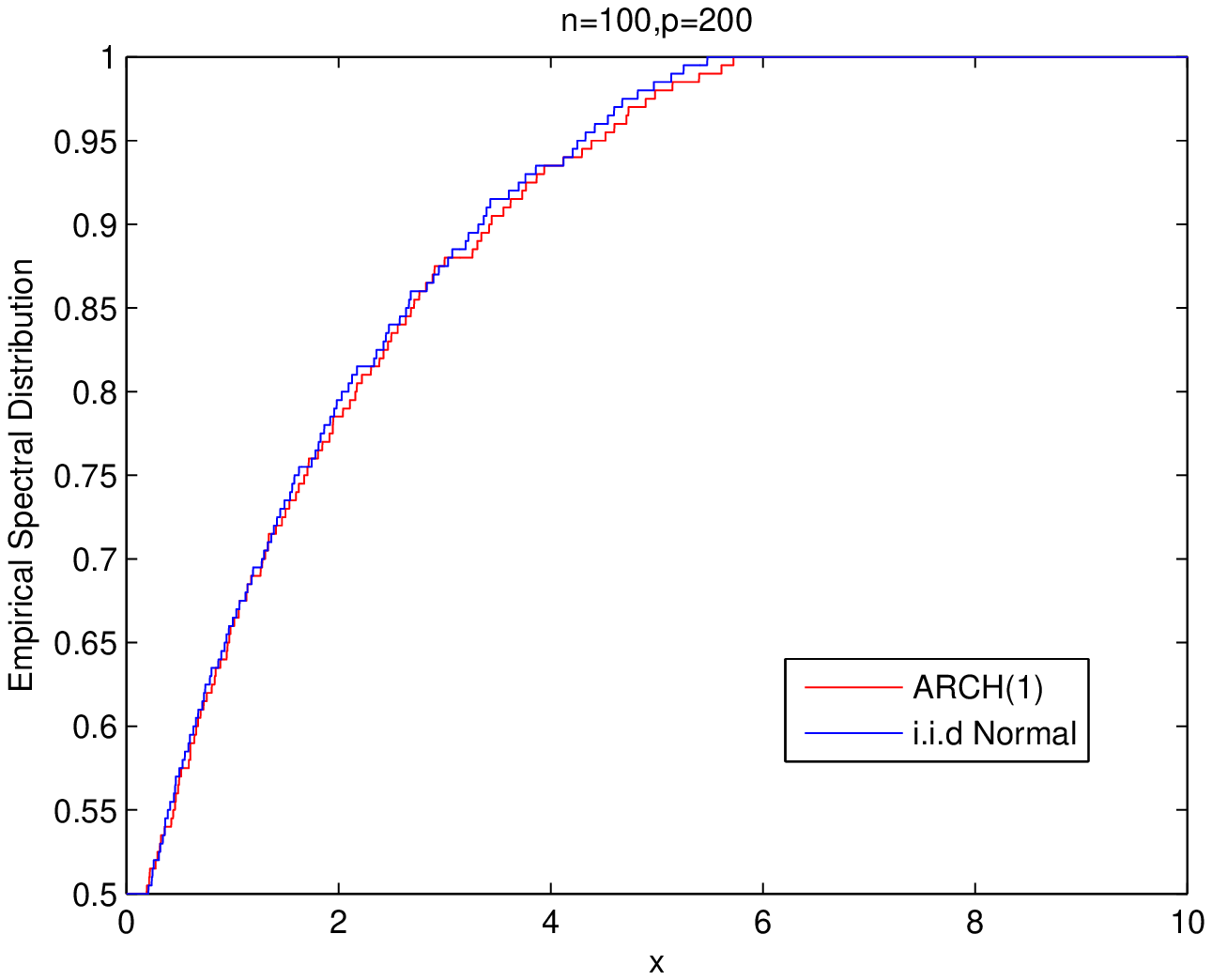}
\includegraphics[scale=0.45]{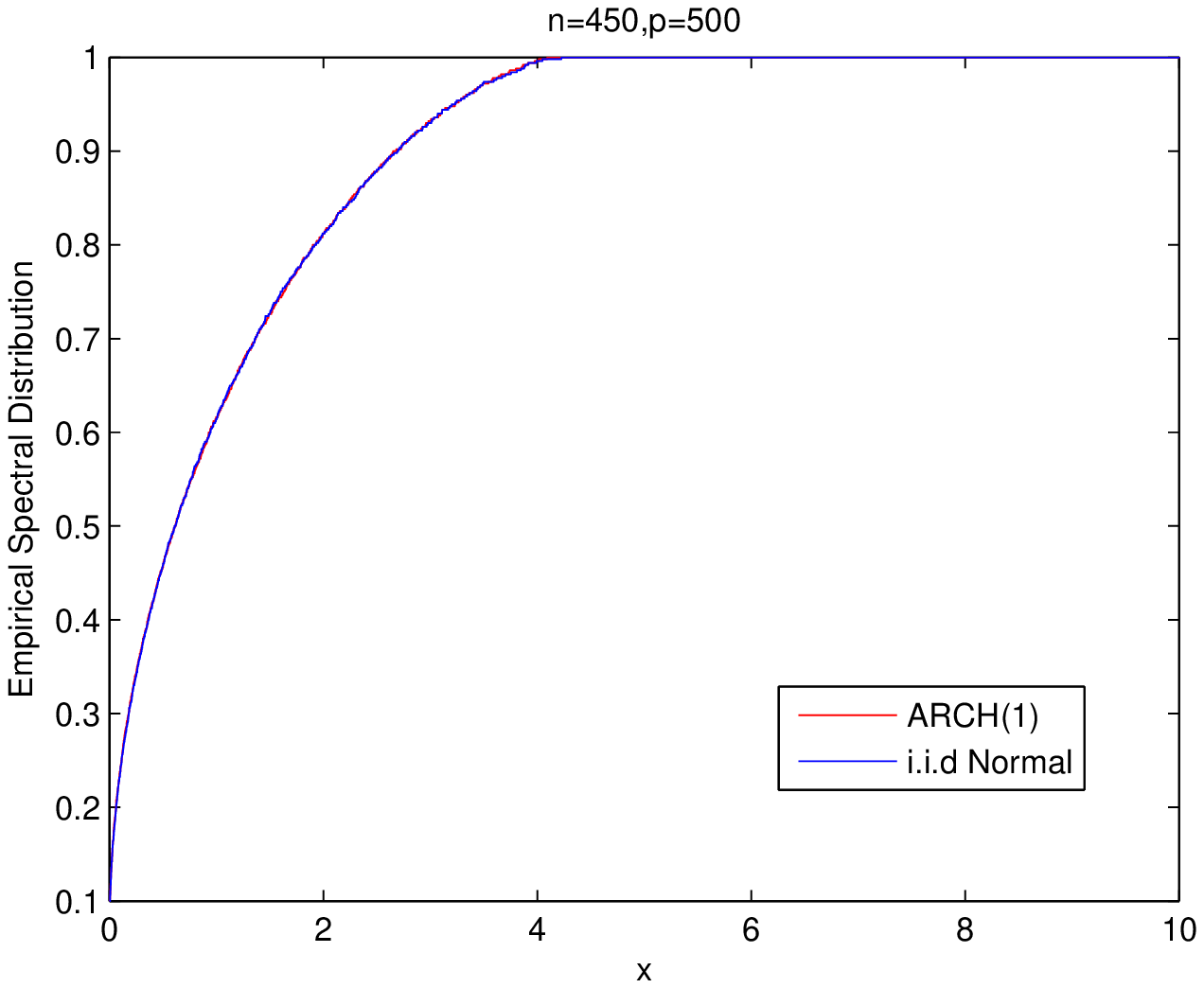}
\\
\includegraphics[scale=0.45]{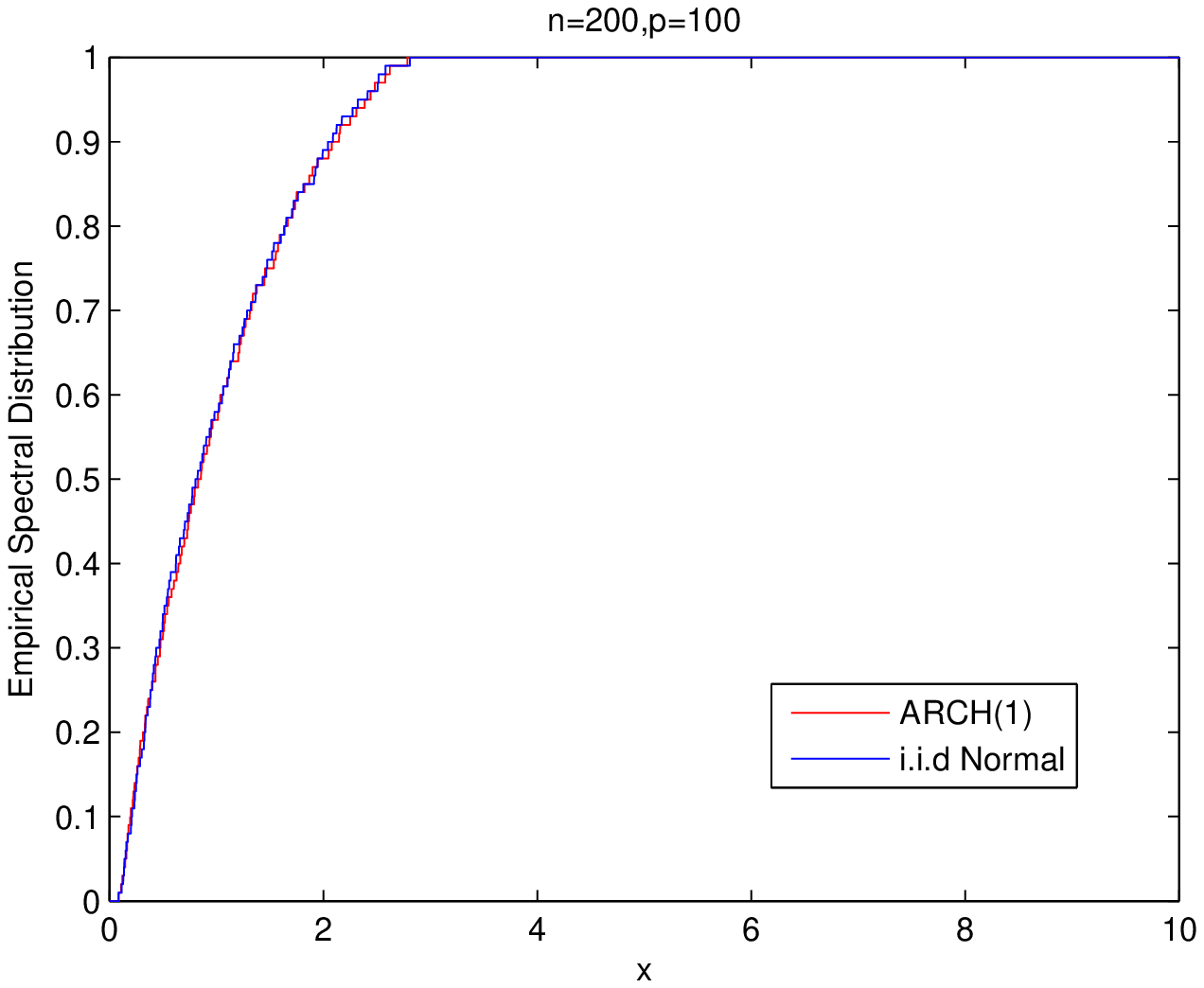}
\includegraphics[scale=0.45]{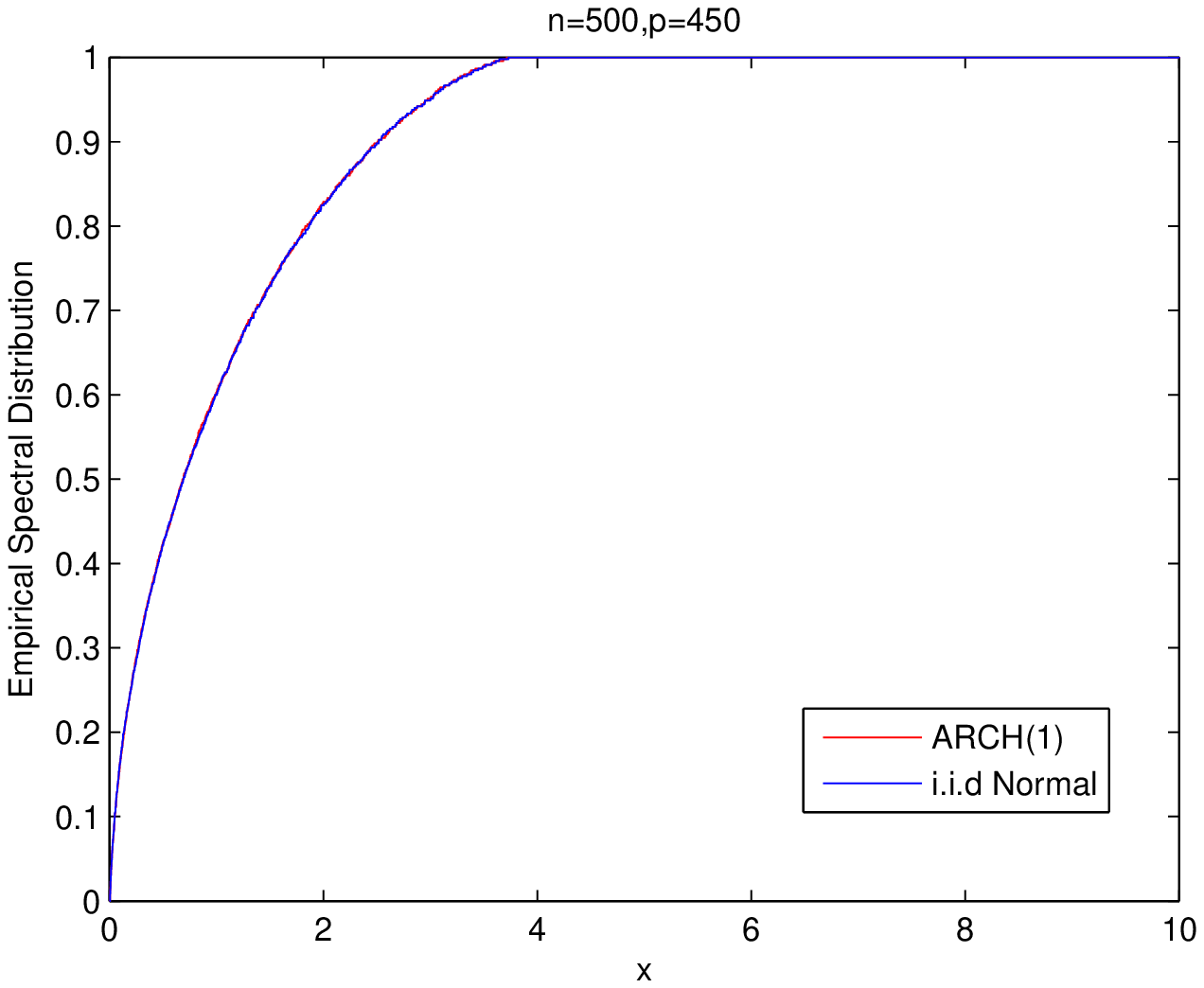}
\end{figure}

\begin{figure}[h]
\caption{P-values of the proposed test for daily closed stock prices from NYSE}\label{fig1}
\centering
\includegraphics[scale=0.45]{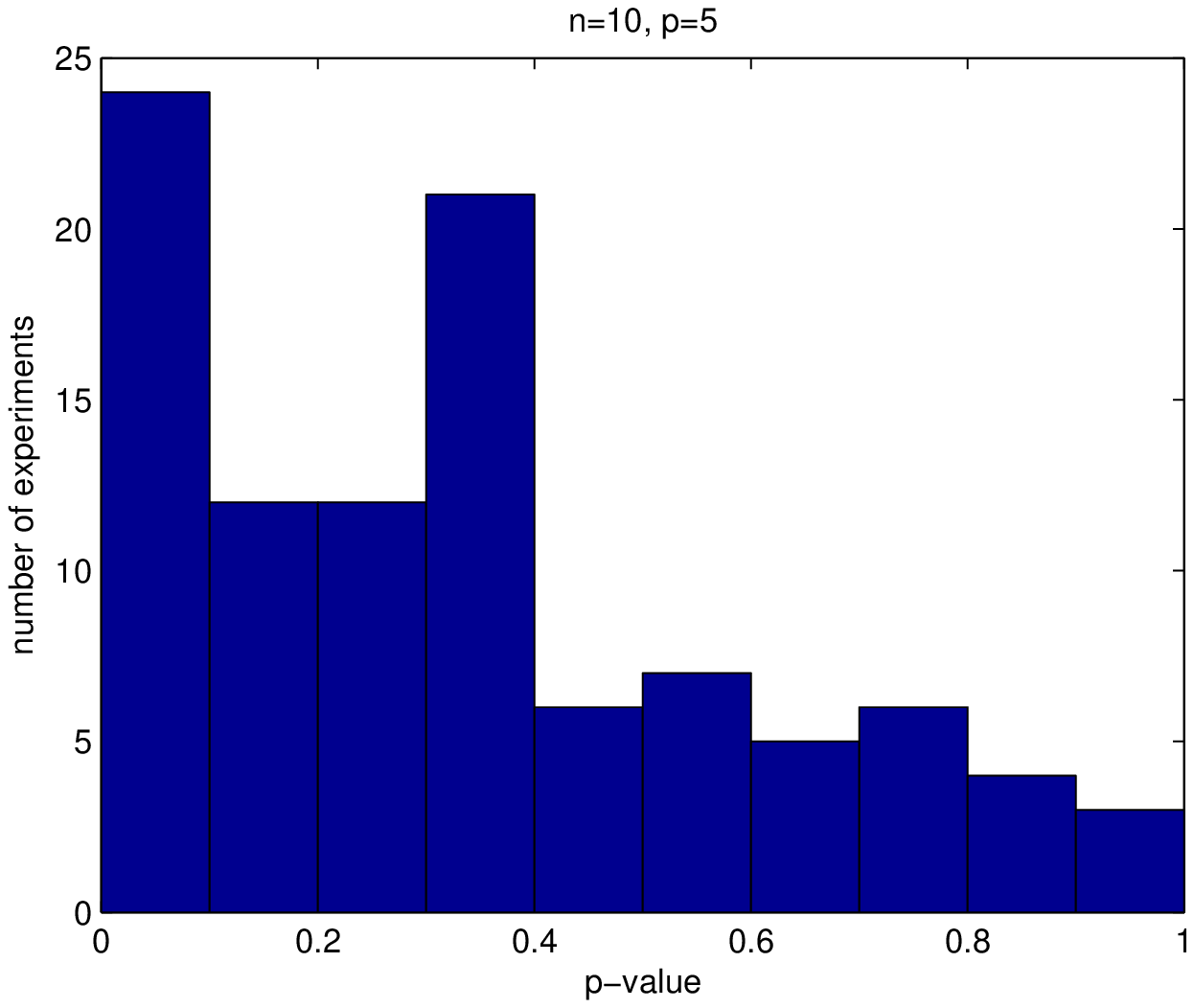}
\includegraphics[scale=0.45]{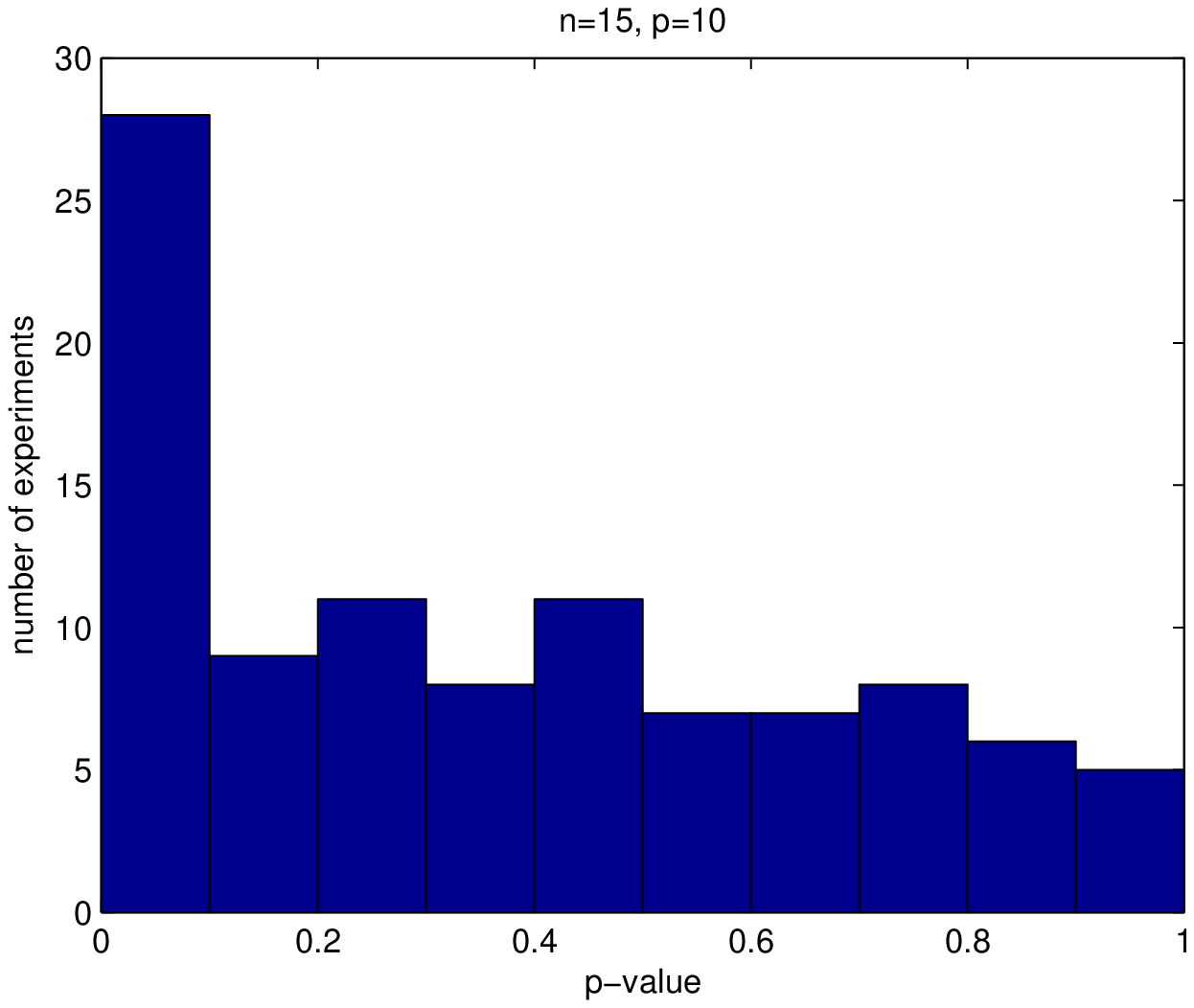}
\\
\includegraphics[scale=0.45]{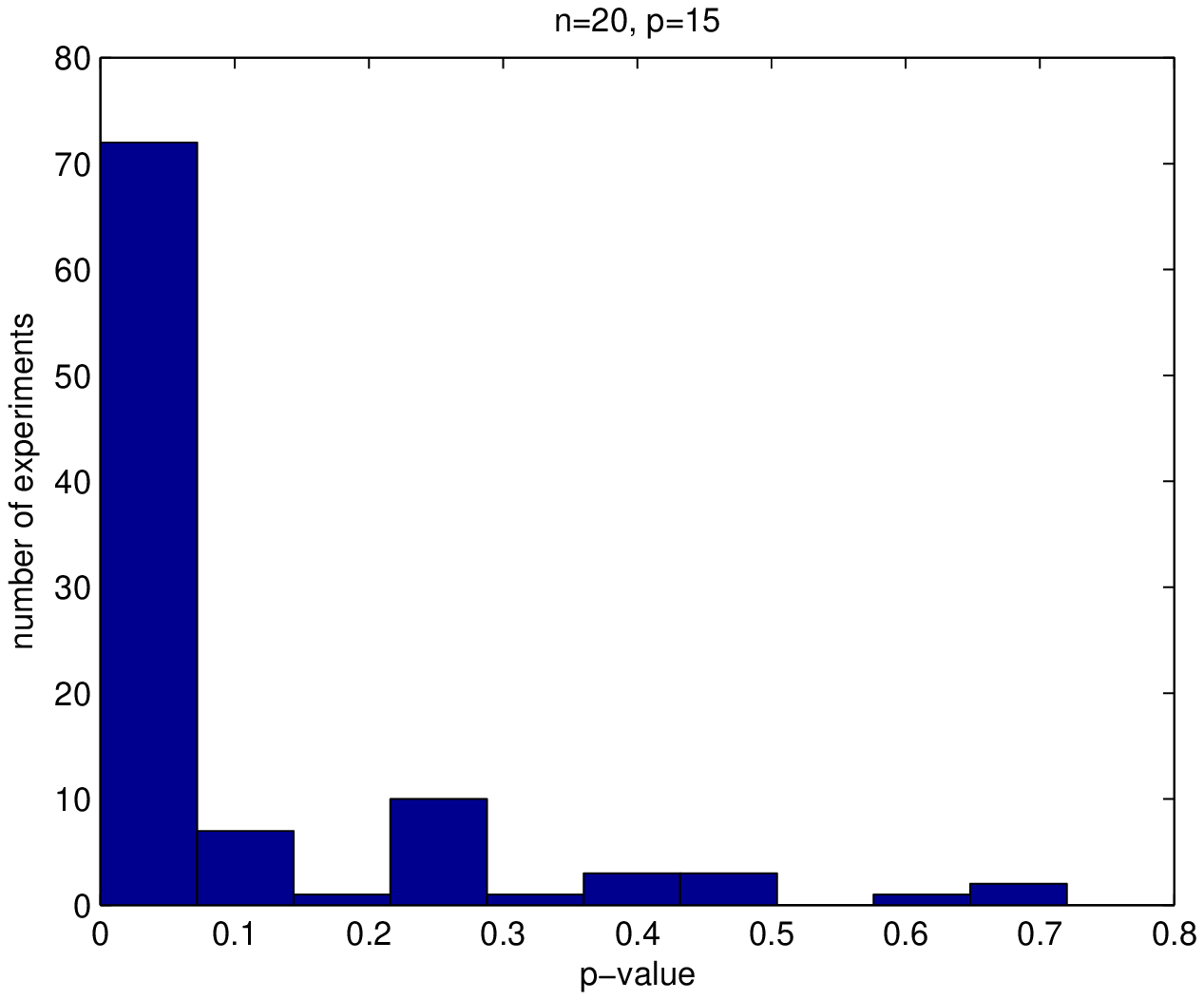}
\includegraphics[scale=0.45]{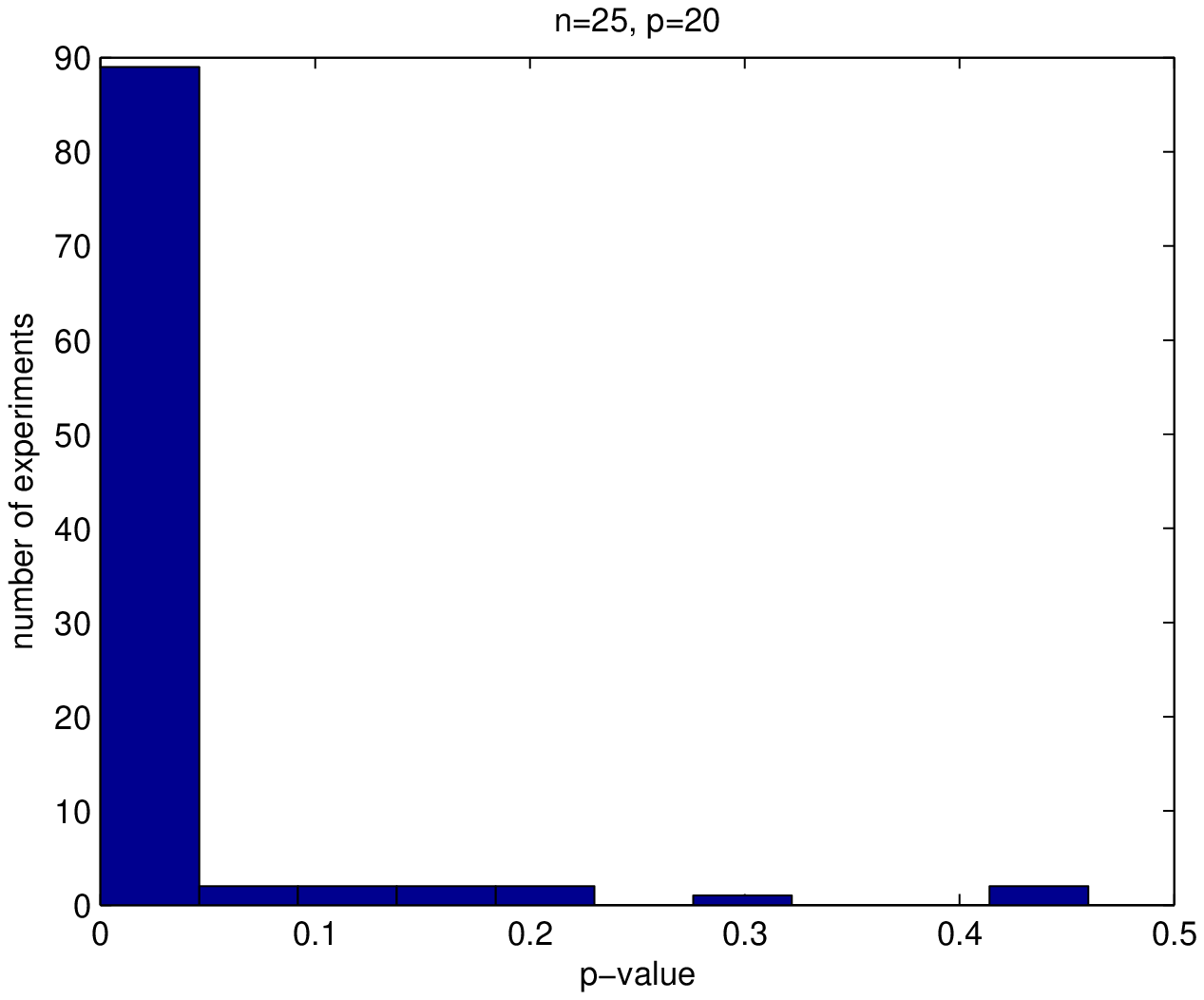}
\\
\includegraphics[scale=0.45]{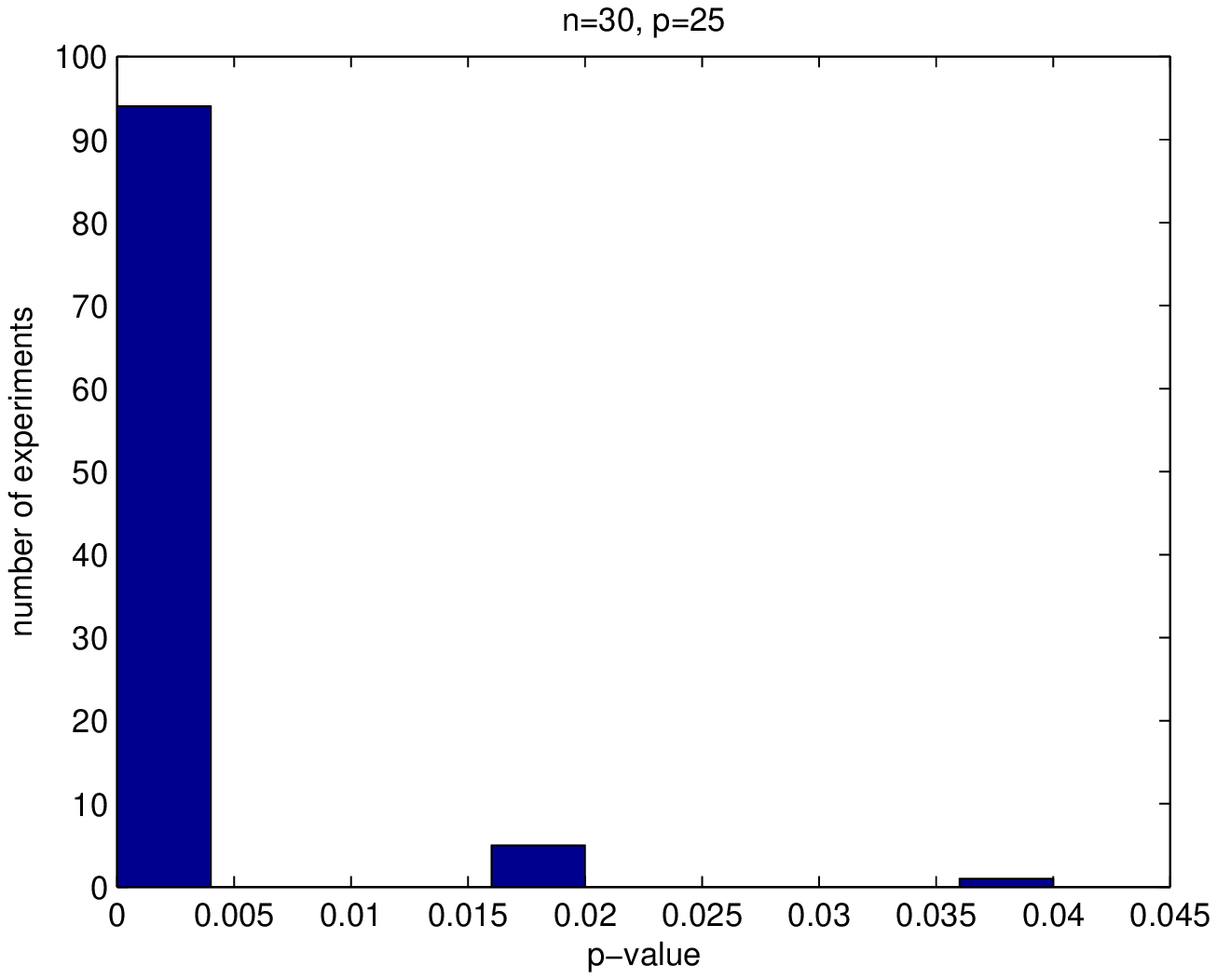}
\includegraphics[scale=0.45]{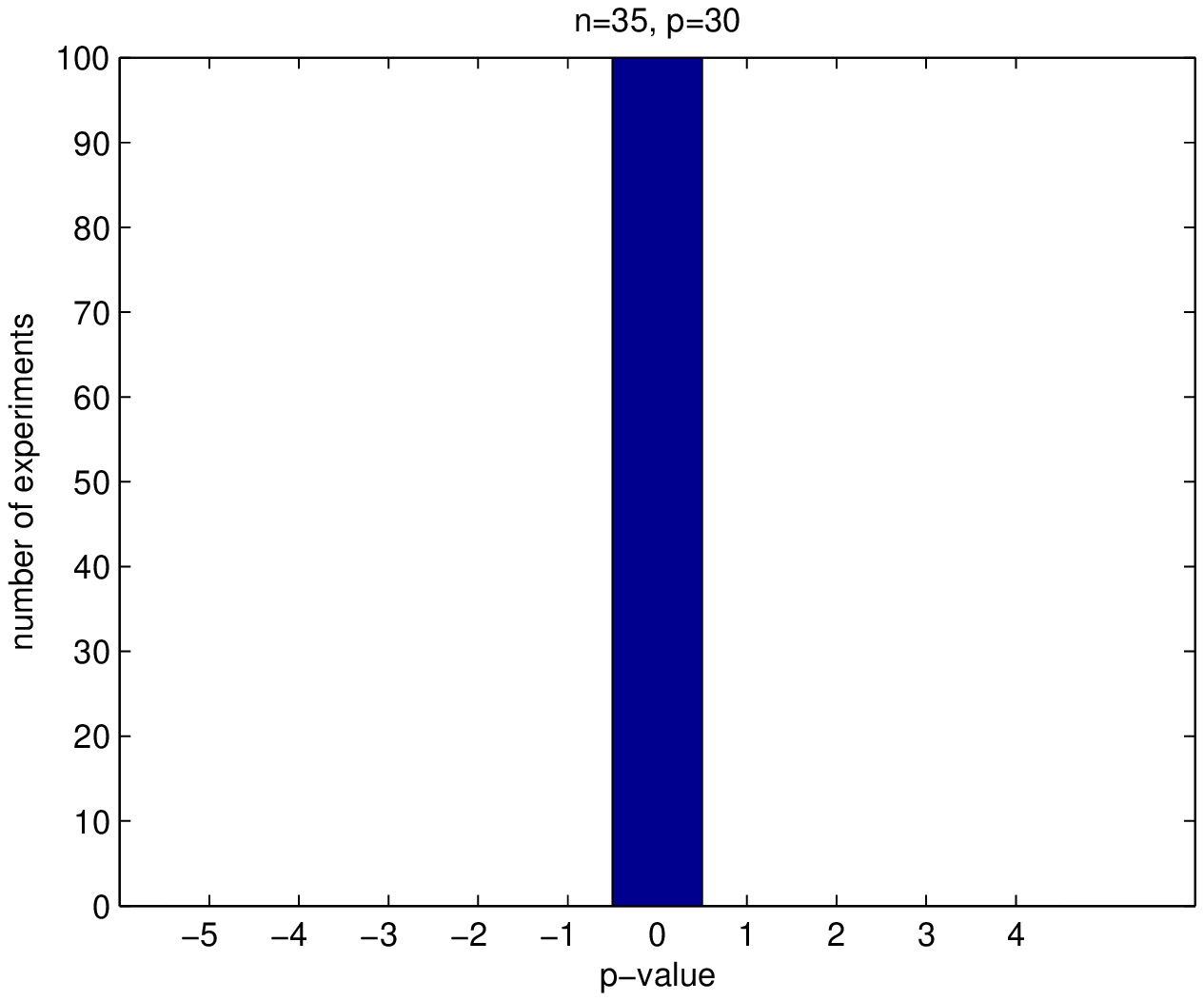}
\\
{\footnotesize *These are P-values for $p$ companies from NYSE,  each of which has $n$ closed stock prices during the period $1990.1.1-2002.1.1$. The number of repeated experiments are $100$. All the closed stock prices are from WRDS database.}
\end{figure}

\begin{figure}[h]
\caption{P-values of the proposed test for daily closed stock prices from transportation sections of NYSE}\label{fig2}
\centering
\includegraphics[scale=0.45]{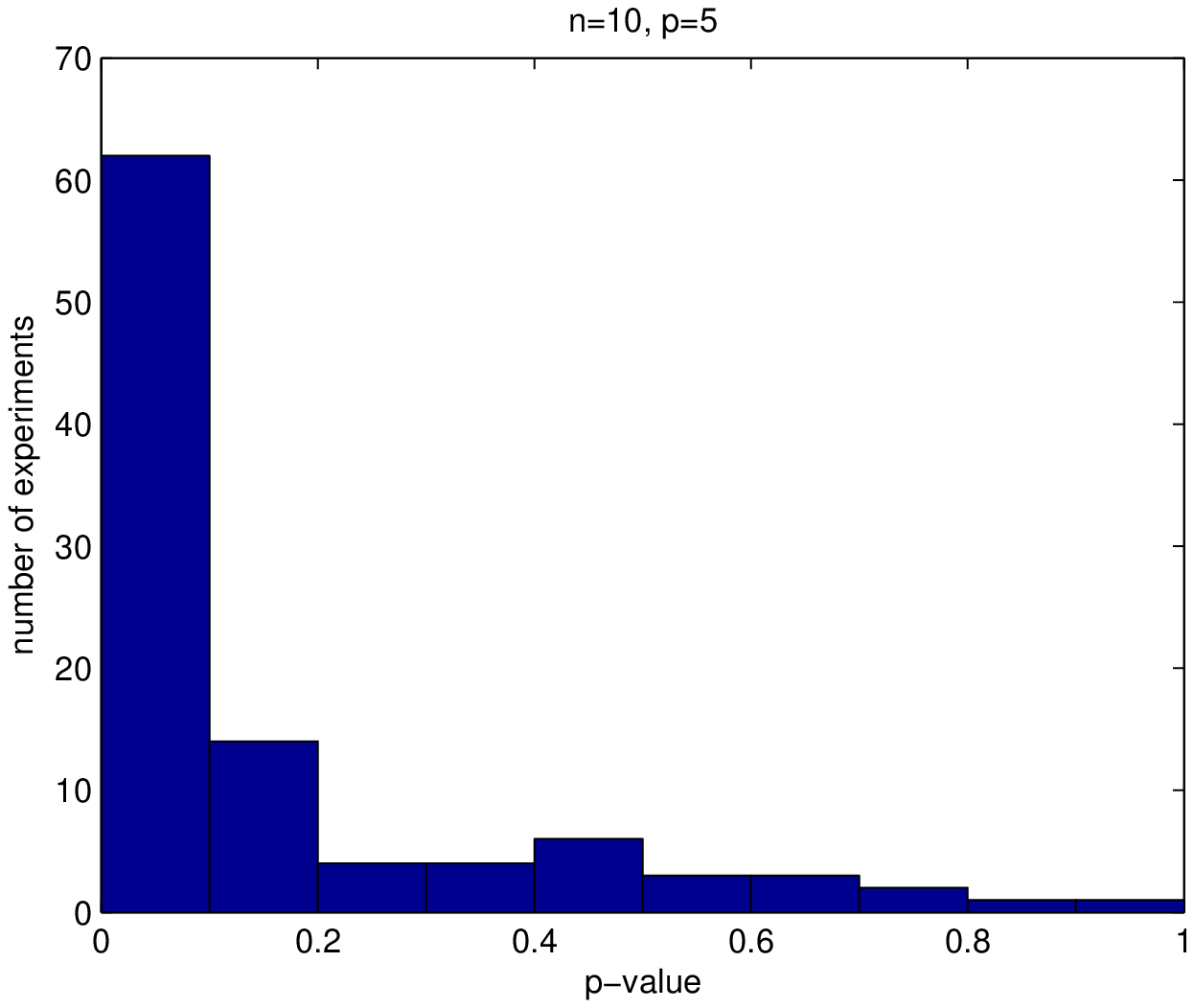}
\\
\includegraphics[scale=0.35]{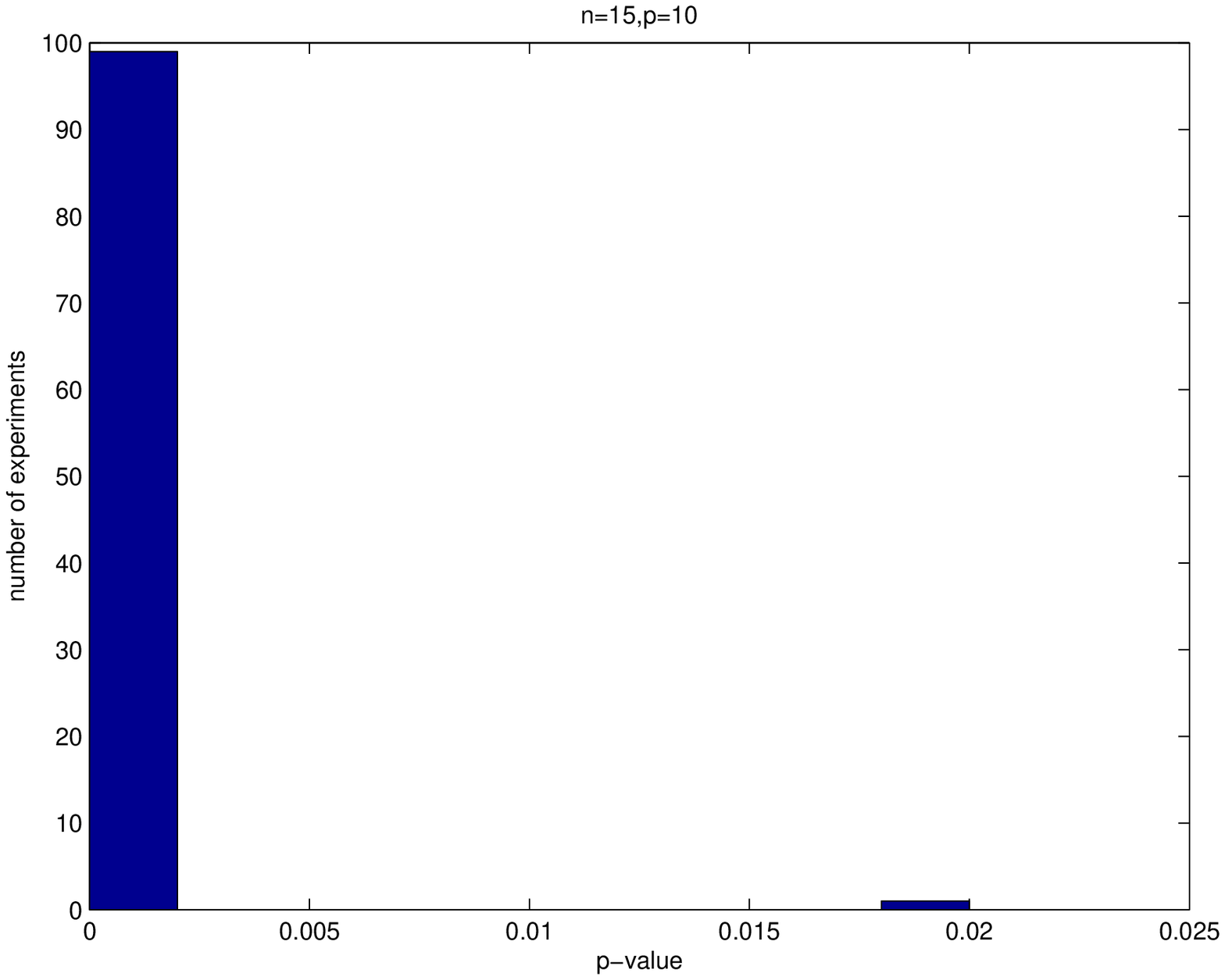}
\includegraphics[scale=0.35]{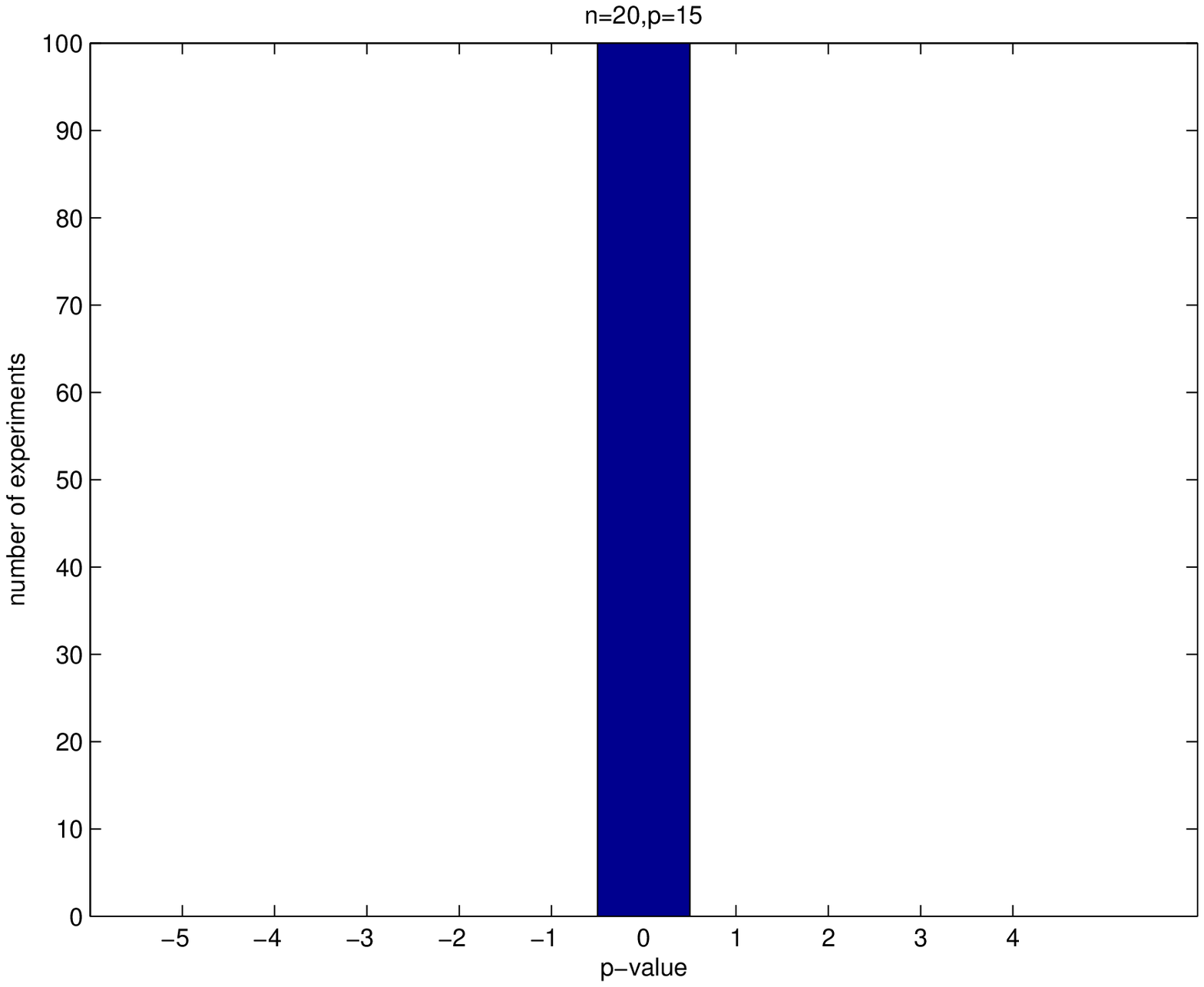}
\\
{\footnotesize *These are P-values for $p$ companies from transportation sections of NYSE,  each of which has $n$ closed stock prices during the period $1990.1.1-2002.1.1$. The number of repeated experiments are $100$. All the closed stock prices are from WRDS database.}
\end{figure}

\begin{table}[h]
{\tiny \caption{\label{tb1}Empirical sizes and powers of the proposed test at significant level $0.05$ for standardized normal distributed random vectors.}}
\begin{center}
{\tiny
\begin{tabular}{ccccccccccccc}
\hline
&\multicolumn{11}{c}{$p$}  \\
\hline
n  &   5  & 10   & 20   & 30   &  40  & 50   & 60   & 70 & 80 & 90 & 100 \\
\hline
&\multicolumn{10}{c}{Empirical sizes}\\
5&0.069&	0.064&	0.027&	0.039&	0.034&	0.052&	0.035&	0.033&	0.037&	0.043&	0.047\\
10&0.049&	0.059&	0.048&	0.046&	0.055&	0.044&	0.043&	0.048&	0.048&	0.038&	0.046\\
20&0.045&	0.042&	0.047&	0.051&	0.052&	0.052&	0.051&	0.041&	0.042&	0.046&	0.041\\
30&0.034&	0.035&	0.052&	0.053&	0.068&	0.063&	0.050&	0.049&	0.047&	0.040&	0.046\\
40&0.055&	0.050&	0.052&	0.061&	0.056&	0.038&	0.042&	0.064&	0.054&	0.044&	0.066\\
50&0.038&	0.041&	0.048&	0.044&	0.050&	0.052&	0.050&	0.058&	0.046&	0.053&	0.041\\
60&0.064&	0.061&	0.043&	0.049&	0.047&	0.049&	0.050&	0.049&	0.059&	0.046&	0.048\\
70&0.047&	0.042&	0.051&	0.039&	0.052&	0.055&	0.058&	0.044&	0.037&	0.038&	0.049\\
80&0.028&	0.032&	0.033&	0.041&	0.052&	0.054&	0.047&	0.042&	0.047&	0.046&	0.040\\
90&0.042&	0.037&	0.032&	0.048&	0.054&	0.044&	0.053&	0.041&	0.037&	0.051&	0.042\\
100&0.037&	0.046&	0.036&	0.035&	0.045&	0.029&	0.045&	0.056&	0.048&	0.047&	0.056\\
&\multicolumn{11}{c}{Empirical powers}  \\
5&0.044&	0.053&	0.072&	0.070&	0.065&	0.048&	0.053&	0.063&	0.057&	0.053&	0.080\\
10&0.054&	0.065&	0.046&	0.051&	0.052&	0.049&	0.049&	0.037&	0.049&	0.060&	0.062\\
20&0.056&	0.066&  0.091&	0.059&	0.062&	0.051&	0.045&	0.075&	0.078&	0.078&	0.063\\
30&0.051&	0.062&	0.054&	0.068&	0.051&	0.081&	0.092&	0.087&	0.115&	0.128&	0.148\\
40&0.041&	0.036&	0.052&	0.055&	0.126&	0.183&	0.144&	0.208&	0.209&	0.276&	0.254\\
50&0.068&	0.051&	0.054&	0.066&	0.151&	0.262&	0.329&	0.420&	0.474&	0.432&	0.574\\
60&0.046&	0.027&	0.055&	0.098&	0.154&	0.310&	0.455&	0.592&	0.596&	0.733&	0.715\\
70&0.050&	0.044&	0.018&	0.110&	0.211&	0.379&	0.565&	0.747&	0.793&	0.846&	0.841\\
80&0.053&	0.046&	0.036&	0.102&	0.233&	0.439&	0.687&	0.892&	0.906&	0.944&	0.977\\
90&0.028&	0.040&	0.054&	0.122&	0.226&	0.484&	0.752&	0.878&	0.963&	0.980&	0.999\\
\hline
\end{tabular}}
\\ \end{center}\medskip
{\footnotesize *The powers are under the alternative hypothesis that the population covariance matrix is $\Sigma=0.05I_p+0.95\textbf{1}_p\textbf{1}_p^{'}$.}
\end{table}

\begin{table}[h]
{\small \caption{\label{tb2}Empirical sizes and powers of the LRT at significant level $0.05$ for standardized normal distributed random vectors.}}
\begin{center}
{\tiny
\begin{tabular}{ccccccccccccc}
\hline
&\multicolumn{11}{c}{$p$}  \\
\hline
n  &   5  & 10   & 20   & 30   &  40  & 50   & 60   & 70 & 80 & 90 & 100 \\
\hline
&\multicolumn{11}{c}{Empirical sizes}\\
5&0.005&	0.475&	0.495&	0&	0&	0&	0&	0&	0&	0&	0\\
10&0.030&	0.115&	0.504&	0.454&	0&	0&	0&	0&	0&	0&	0\\
20&0.047&	0.030&	0.522&	1&	0.500&	0&	0&	0&	0&	0&	0\\
30&0.040&	0.032&	0.076&	0.873&	1&	1&	0&	0&	0&	0&	0\\
40&0.044&	0.034&	0.057&	0.148&	0.982&	1&	1&	1&	0&	0&	0\\
50&0.057&	0.034&	0.044&	0.062&	0.299&	1&	1&	1&	1&	1&	0\\
60&0.049&	0.053&	0.040&	0.065&	0.101&	0.578&	1&	1&	1&	1&	1\\
70&0.038&	0.060&	0.049&	0.054&	0.108&	0.229&	0.784&	1&	1&	1&	1\\
80&0.055&	0.040&	0.043&	0.072&	0.079&	0.147&	0.442&	0.950&	1&	1&	1\\
90&0.054&	0.061&	0.047&	0.033&	0.048&	0.129&	0.233&	0.691&	0.991&	1&	1\\
100&0.055&	0.048&	0.066&	0.069&	0.076&	0.085&	0.185&	0.402&	0.844&	0.997&	1\\
&\multicolumn{11}{c}{Empirical powers}  \\
5&0.005&	0.504&	0.505&	0&	0&	0&	0&	0&	0&	0&	0\\
10&0.033&	0.125&	0.506&	0.485&	0&	0&	0&	0&	0&	0&	0\\
20&0.051&	0.049&	0.587&	1&	0.496&	0&	0&	0&	0&	0&	0\\
30&0.056&	0.062&	0.121&	0.912&	1&	1&	0&	0&	0&	0&	0\\
40&0.061&	0.077&	0.139&	0.311&	0.990&	1&	1&	1&	0&	0&	0\\
50&0.092&	0.090&	0.150&	0.206&	0.576&	1&	1&	1&	1&	1&	0\\
60&0.088&	0.142&	0.220&	0.266&	0.420&	0.849&	1&	1&	1&	1&	1\\
70&0.102&	0.170&	0.222&	0.320&	0.463&	0.691&	0.984&	1&	1&	1&	1\\
80&0.133&	0.169&	0.246&	0.392&	0.478&	0.666&	0.883&	0.994&	1&	1&	1\\
90&0.122&	0.190&	0.331&	0.404&	0.526&	0.666&	0.819&	0.983&	1&	1&	1\\
100&0.153&	0.232&	0.411&	0.550&	0.658&	0.713&	0.842&	0.952&	0.995&	1&	1\\
\hline
\end{tabular}}
\\ \end{center}\medskip
{\footnotesize *The powers are under the alternative hypothesis that the population covariance matrix is $\Sigma=0.05I_p+0.95\textbf{1}_p\textbf{1}_p^{'}$.}
\end{table}

\begin{table}[h]
{\small \caption{\label{tb3}Empirical sizes and powers of the proposed test at $0.05$ significance level for standardized gamma distributed random vectors.}}
\begin{center}
{\tiny
\begin{tabular}{ccccccccccccc}
\hline
&\multicolumn{11}{c}{$p$}  \\
\hline
n  &   5  & 10   & 20   & 30   &  40  & 50   & 60   & 70 & 80 & 90 & 100 \\
\hline
&\multicolumn{11}{c}{Empirical sizes}\\
5&0.089&	0.078&	0.068&	0.059&	0.060&	0.063&	0.061&	0.047&	0.050&	0.061&	0.044\\
10&0.066&	0.075&	0.068&	0.054&	0.046&	0.051&	0.047&	0.048&	0.050&	0.042&	0.052\\
20&0.057&	0.058&	0.072&	0.062&	0.053&	0.054&	0.053&	0.052&	0.049&	0.051&	0.048\\
30&0.069&	0.069&	0.069&	0.076&	0.058&	0.056 &       0.054&	0.047&	0.067&	0.066&	0.054\\
40&0.061&	0.049&	0.058&	0.040&	0.065&	0.048&	0.063&	0.065&	0.068&	0.047&	0.067\\
50&0.053&	0.054&	0.055&	0.057&	0.059&        0.048&	0.067&	0.066&	0.043&	0.059&	0.053\\
60&0.059&	0.052&	0.057&	0.060&	0.052&	0.067 &        0.058&	0.064&	0.064&	0.061&	0.069\\
70&0.044&	0.050&	0.064&	0.055&	0.071&	0.054&	0.067&	0.064&	0.051&	0.077&	0.048\\
80&0.045&	0.050&	0.061&	0.043&	0.071&	0.055&	0.071&	0.053&	0.056&	0.070 &       0.060\\
90&0.041&	0.067&	0.034&	0.056&	0.049&	0.054&	0.050&	0.060&	0.047&	0.060&	0.058\\
100&0.070&	0.045&	0.059&	0.055&	0.047&	0.062&	0.069&	0.057&	0.056&	0.060&	0.061\\
&\multicolumn{11}{c}{Empirical powers}  \\
5&0.334&	0.575&	0.853&	0.944&	0.983&	0.989&	0.998&	0.997&	0.999&	0.998&	0.998\\
10&0.513&	0.838&	0.979&	0.999&	1&	1&	1&	1&	1&	1&	1\\
20&0.721&	0.970&	0.999&	1&	1&	1&	1&	1&	1&	1&	1\\
30&0.834&	0.998&	1&	1&	1&	1&	1&	1&	1&	1&	1\\
40&0.914&	1&	1&	1&	1&	1&	1&	1&	1&	1&	1\\
50&0.952&	1&	1&	1&	1&	1&	1&	1&	1&	1&	1\\
60&0.991&	1&	1&	1&	1&	1&	1&	1&	1&	1&	1\\
70&0.992&	1&	1&	1&	1&	1&	1&	1&	1&	1&	1\\
80&0.998&	1&	1&	1&	1&	1&	1&	1&	1&	1&	1\\
90&0.999&        1&	1&	1&	1&	1&	1&	1&	1&	1&	1\\
100&1&	1&	1&	1&	1&	1&	1&	1&	1&	1&	1\\

\hline
\end{tabular}}
\\ \end{center}\medskip
{\footnotesize *The powers are under the alternative hypothesis that the population covariance matrix is $\Sigma=0.05I_p+0.95\textbf{1}_p\textbf{1}_p^{'}$.}
\end{table}

\begin{table}[h]
{\small \caption{\label{tb4}Empirical powers of the proposed test at $0.05$ significance level for MA(1) model.}}
\begin{center}
{\tiny
\begin{tabular}{ccccccccccccc}
\hline
&\multicolumn{11}{c}{$p$}  \\
\hline
n  &   5  & 10   & 20   & 30   &  40  & 50   & 60   & 70 & 80 & 90 & 100 \\
\hline
5 & 0.096 & 0.097 & 0.097& 0.111& 0.210 & 0.214 &0.198&	0.221&	0.201&	0.223&	0.199\\
10&0.082&0.090&	0.173&	0.227&	0.374&	0.715&	0.722&	0.809&	0.903&	0.913&	0.934\\
20&0.099&0.165&	0.400&	0.597&	0.683&	0.822&	0.951&	0.994&	1&	1&	1\\
30&0.067&0.121&	0.611&  0.733&  0.803&	0.986&	1&	1&	1&	1&	1\\
40&0.091&0.321&	0.653&  0.938&  0.968&	1&	1&	1&	1&	1&	1\\
50&0.139&0.416&	0.910&  0.956&  0.998&	1&	1&	1&	1&	1&	1\\
60&0.117&0.412&	0.918&0.994&1&	1&	1&	1&	1&	1&	1\\
70&0.146&0.535&	0.983&	1&	1&	1&	1&	1&	1&	1&	1\\
80&0.138&	0.620&	0.990&	1&	1&	1&	1&	1&	1&	1&	1\\
90&0.239&	0.805&	0.999&	1&	1&	1&	1&	1&	1&	1&	1\\
100&0.132&	0.761&	1&	1&	1&	1&	1&	1&	1&	1&	1\\
\hline
\end{tabular}}
\\ \end{center}\medskip
\end{table}

\begin{table}[h]
{\small \caption{\label{tb6AR}Empirical powers of the proposed test at $0.05$ significance level for AR(1) model.}}
\begin{center}
{\tiny
\begin{tabular}{cccccccccccc}
\hline
&\multicolumn{11}{c}{$p$}  \\
\hline
n  &   5  & 10   & 20   & 30   &  40  & 50   & 60   & 70 & 80 & 90 & 100 \\
\hline
5 & 0.130 & 0.094 & 0.264 & 0.253 &	0.181 &	0.169 &	0.191 &	0.134 &	0.144 &	0.129 &	0.113\\
10& 0.167 &	0.289 &	0.482 &	0.640 &	0.668 &	0.609 &	0.563 &	0.524 &	0.427 &	0.384 &	0.321\\
20& 0.230 &	0.544 &	0.878 &	0.271 &	0.954 &	0.994 &	0.999 &	0.997 &	0.996 &	0.984 &	0.974\\
30& 0.205 &	0.602 &	0.993 &	0.999 &	0.671 &	0.936 &	1 &	1 &	1 &	1 &	1\\
40& 0.344 &	0.916 &	0.998 &	1 &	1 &	0.982 &	1 &	1 &	1 &	1 &	1\\
50& 0.541 &	0.984 &	1 &	1 &	1 &	1 &	1 &	1 &	1 &	1 &	1\\
60& 0.571 &	0.986 &	1 &	1 &	1 &	1 &	1 &	1 &	1 &	1 &	1\\
70& 0.672 &	0.997 &	1 &	1 &	1 &	1 &	1 &	1 &	1 &	1 &	1\\
80& 0.673 &	0.999 &	1 &	1 &	1 &	1 &	1 &	1 &	1 &	1 &	1\\
90& 0.882 &	1 &	1 &	1 &	1 &	1 &	1 &	1 &	1 &	1 &	1\\
100&0.801 &	1 &	1 &	1 &	1 &	1 &	1 &	1 &	1 &	1 &	1\\
\hline
\end{tabular}}
\\ \end{center}\medskip
\end{table}

\begin{table}[h]
{\small \caption{\label{tb6SMA}Empirical powers of the proposed test at $0.05$ significance level for SMA(1) model.}}
\begin{center}
{\tiny
\begin{tabular}{cccccccccccc}
\hline
&\multicolumn{11}{c}{$p$}  \\
\hline
n  &   5  & 10   & 20   & 30   &  40  & 50   & 60   & 70 & 80 & 90 & 100 \\
\hline
5&0.330&	0.431&	0.782&	0.999&	1&	1&	1&	1&	1&	1&	1\\
10&0.647&	1&	1&	1&	1&	1&	1&	1&	1&	1&	1\\
20&0.967&	1&	1&	1&	1&	1&	1&	1&	1&	1&	1\\
30&0.962&	1&	1&	1&	1&	1&	1&	1&	1&	1&	1\\
40&0.998&	1&	1&	1&	1&	1&	1&	1&	1&	1&	1\\
50&1&	1&	1&	1&	1&	1&	1&	1&	1&	1&	1\\
60&1&	1&	1&	1&	1&	1&	1&	1&	1&	1&	1\\
70&1&	1&	1&	1&	1&	1&	1&	1&	1&	1&	1\\
80&1&	1&	1&	1&	1&	1&	1&	1&	1&	1&	1\\
90&1&	1&	1&	1&	1&	1&	1&	1&	1&	1&	1\\
100&1&	1&	1&	1&	1&	1&	1&	1&	1&	1&	1\\
\hline
\end{tabular}}
\\ \end{center}\medskip
\end{table}

\begin{table}[h]
{\small \caption{\label{tb8}Empirical sizes and powers of the proposed test at $0.05$ significance level for the general panel data model.}}
\begin{center}
{\tiny
\begin{tabular}{cccccccccccc}
\hline
&\multicolumn{11}{c}{$p$}  \\
\hline
n  &   5  & 10   & 20   & 30   &  40  & 50   & 60   & 70 & 80 & 90 & 100 \\
\hline
&\multicolumn{10}{c}{Empirical sizes}\\
5 & 0.038 & 0.045 & 0.047 & 0.057 & 0.058 & 0.067 & 0.069 & 0.069 & 0.072 & 0.076 &0.074\\
10 & 0.041 & 0.042 & 0.046 & 0.049 & 0.056 & 0.050 & 0.065 & 0.047 & 0.068 & 0.063 &0.069\\
20 & 0.035 & 0.042 & 0.049 & 0.056 & 0.052 & 0.046 & 0.063 & 0.043 & 0.069 & 0.055 &0.057\\
30 & 0.043 & 0.049 & 0.043 & 0.059 & 0.048 & 0.068 & 0.059 & 0.057 & 0.040 & 0.055 &0.047\\
40& 0.048 & 0.052 & 0.043 & 0.060 & 0.057 & 0.046 & 0.049 & 0.054 & 0.046 & 0.058 &0.061\\
50& 0.057 & 0.048 & 0.046 & 0.058 & 0.052 & 0.055 & 0.048 & 0.049 & 0.050 & 0.040 &0.041\\
60& 0.058 & 0.056 & 0.055 & 0.048 & 0.047 & 0.045 & 0.053 & 0.066 & 0.058 & 0.049 & 0.050\\
70& 0.062 & 0.060 & 0.059 & 0.056 & 0.049 & 0.057 & 0.049 & 0.068 & 0.052 & 0.036 &0.043\\
80& 0.071 & 0.063 & 0.067 & 0.047 & 0.048 & 0.058 & 0.056 & 0.044 & 0.059 & 0.057 & 0.055\\
90& 0.065 & 0.068 & 0.065 & 0.048 & 0.053 & 0.048 & 0.056 & 0.048 & 0.048 & 0.066 & 0.060\\
100&0.037&	0.046&	0.036&	0.035&	0.045&	0.043&	0.045&	0.056&	0.048&	0.047&	0.055\\
&\multicolumn{11}{c}{Empirical powers}  \\
5&0.150&	0.238&	0.345&	0.417&	0.484&	0.529&	0.549&	0.615&	0.611&	0.668&	0.692\\
10&0.125&	0.247&	0.452 &       0.526&	0.568&	0.633&	0.669&	0.737&	0.765&	0.751&	0.800\\
20&0.206&	0.343&	0.493&	0.615&	0.673&	0.752&	0.788&	0.813&	0.860&	0.864&	0.876\\
30&0.111&	0.404&	0.535&	0.684&	0.757&	0.756&	0.855&	0.875&	0.882&	0.909&	0.953\\
40&0.308&	0.393&	0.605&	0.698&	0.786&	0.820&	0.878&	0.898&	0.944&	0.959&	0.953\\
50&0.207&	0.450&	0.603&	0.718&	0.815&	0.889&	0.923&	0.938&	0.966&	0.973&	0.980\\
60&0.268&	0.430&	0.594&	0.780&	0.826&	0.918&	0.913&	0.926&	0.974&	0.976&	0.984\\
70&0.144&	0.434&	0.649&	0.798&	0.888&	0.883&	0.944&	0.968&	0.971&	0.982&	0.996\\
80&0.171&	0.454&	0.678&	0.796&	0.872&	0.921&	0.938&	0.967&	0.989&	0.992&	0.995\\
90&0.204&	0.431&	0.683&	0.834&	0.874&	0.916&	0.963&	0.985&	0.985&	0.994&	0.994\\
100&0.291&	0.398&	0.687&	0.836&	0.884&	0.931&	0.973&	0.987&	0.992&	0.994&	1\\
\hline
\end{tabular}}
\\ \end{center}\medskip
\end{table}

\begin{table}[h]
{\small \caption{\label{tb5}Empirical powers of the proposed test at $0.05$ significance level for nonlinear MA model.}}
\begin{center}
{\tiny
\begin{tabular}{ccccccccccccc}
\hline
&\multicolumn{11}{c}{$p$}  \\
\hline
n  &   5  & 10   & 20   & 30   &  40  & 50   & 60   & 70 & 80 & 90 & 100 \\
\hline
5&0.033&	0.023&	0.004&	0.008&	0.005&	0.007&	0.008&	0.007&	0.009&	0.014&	0.070\\
10&0.618&	0.543&	0.374&	0.220&	0.108&	0.037&	0.011&	0&	0&	0&	0\\
20&0.804&	0.744&	0.703&	0.614&	0.581&	0.511&	0.447&	0.340&	0.306&	0.257&	0.216\\
30&0.854&	0.815&	0.777&	0.779&	0.780&	0.740&	0.662&	0.662&	0.597&	0.579&	0.555\\
40&0.878&	0.842&	0.856&	0.856&	0.845&	0.825&	0.779&	0.770&	0.772&	0.702&	0.698\\
50&0.884&	0.899&  0.868&	0.884&	0.864&	0.888&	0.860&	0.875&	0.869&	0.828&	0.820\\
60&0.892&	0.912&	0.882&	0.904&	0.920&	0.923&	0.933&	0.900&	0.892&	0.892&	0.882\\
70&0.896&	0.934&	0.906&	0.927&	0.934&	0.921&	0.952&	0.925&	0.943&	0.917&	0.926\\
80&0.936&	0.886&	0.925&	0.921&	0.952&	0.950&	0.943&	0.943&	0.958&	0.953&	0.936\\
90&0.922&	0.925&	0.939&	0.935&	0.958&	0.959&	0.955&	0.982&	0.962&	0.957&	0.954\\
100&0.926&	0.935&	0.920&	0.937&	0.952&	0.964&	0.970&	0.975&	0.978&	0.970&	0.965\\
\hline
\end{tabular}}
\\ \end{center}\medskip
\end{table}

\begin{table}[h]
{\small \caption{\label{tb6}Empirical powers of the proposed test at $0.05$ significance level for ARCH(1) model.}}
\begin{center}
{\tiny
\begin{tabular}{cccccccccccc}
\hline
&\multicolumn{11}{c}{$p$}  \\
\hline
n  &   5  & 10   & 20   & 30   &  40  & 50   & 60   & 70 & 80 & 90 & 100 \\
\hline
5  & 0.148 & 0.093 &0.073 & 0.076 & 0.058 &	0.099 &	0.113 &	0.088 &	0.106 &	0.112 &	0.141\\
10 & 0.454 & 0.399 &0.315 &	0.215 &	0.128 &	0.121 &	0.112 &	0.121 &	0.093 &	0.095 &	0.092\\
20 & 0.401 & 0.428 &0.508 &	0.464 &	0.444 &	0.468 &	0.436 &	0.338 &	0.356 &	0.337 &	0.302\\
30 & 0.272 & 0.364 &0.616 &	0.712 & 0.753 &	0.759 &	0.793 &	0.743 &	0.638 &	0.639 &	0.598\\
40 & 0.232 & 0.357 &0.572 &	0.823 &	0.790 &	0.874 &	0.915 &	0.885 &	0.876 &	0.855 &	0.860\\
50 & 0.222 & 0.339 &0.622 &	0.757 &	0.891 &	0.969 &	0.957 &	0.975 &	0.967 &	0.957 &	0.975\\
60 & 0.216 & 0.448 &0.617 &	0.862 &	0.901 &	0.976 &	0.983 &	0.987 &	0.992 &	0.997 &	0.997\\
70 & 0.200 & 0.339 &0.592 &	0.864 &	0.931 &	0.982 &	0.993 &	0.998 &	0.997 &	0.998 &	0.996\\
80 & 0.194 & 0.376 &0.566 &	0.824 &	0.950 &	0.967 &	0.997 &	1	  &  1	  &  0.999&	1\\
90 & 0.184 & 0.456 &0.721 &	0.839 &	0.960 &	0.995 &	0.999 &	1	  & 0.999 &	1	  &  1\\
100& 0.128 & 0.306 &0.802 &	0.859 &	0.934 &	0.992 &	1	  &  1	  &  1	  &  1	  &  1\\
\hline
\end{tabular}}
\\ \end{center}\medskip
\end{table}

\begin{table}[h]
{\small \caption{\label{tb7}Empirical powers of the proposed test at $0.05$ significance level for Vandermonde Matrix.}}
\begin{center}
{\tiny
\begin{tabular}{ccccccccccccc}
\hline
&\multicolumn{11}{c}{$p$}  \\
\hline
n  &   10   & 20   & 30   &  40  & 50   & 60   & 70 & 80 & 90 & 100 &120\\
\hline
10&	0.180&	0.197&	0.192&	0.169&	0.202&	0.211&	0.190&	0.185&	0.171&	0.177&0.240\\
20&	0.309&	0.332&	0.356&	0.327&	0.291&	0.303&	0.301&	0.295&	0.321&	0.243&0.478\\
30&	0.324&	0.433&	0.473&	0.413&	0.461&	0.408&	0.445&	0.395&	0.368&	0.397&0.606\\
40&	0.458&	0.512&	0.527&	0.546&	0.533&	0.490&	0.518&	0.498&	0.450&	0.457&0.655\\
50&	0.593&	0.437&	0.540&	0.571&	0.614&	0.569&	0.577&	0.566&	0.565&	0.537&0.764\\
60&	0.504&	0.538&	0.551&	0.567&	0.616&	0.662&	0.588&	0.581&	0.572&	0.607&0.744\\
70&	0.548&	0.526&	0.560&      0.627&	0.668&	0.641&	0.694&      0.707&	0.641&	0.678&0.741\\
80&	0.550&	0.545&	0.580&	0.633&	0.712&	0.719&	0.693&	0.768&	0.729&	0.749&0.805\\
90&	0.589&	0.544&	0.596&	0.667&	0.695&	0.712&	0.743&	0.754&	0.738&	0.728&0.807\\
100&	0.464&	0.549&	0.610&	0.645&	0.704&	0.757&      0.772&         0.751&        0.752 &        0.808&0.928\\
120&	0.633&	0.660&	0.736&	0.737&	0.759&	0.855&	0.854&	0.909&	0.960&	0.999&1\\
\hline
\end{tabular}}
\\ \end{center}\medskip
\end{table}

\end{document}